\documentclass[default,12pt]{sn-jnl}

\usepackage{graphicx}%
\usepackage{multirow}%
\usepackage{amsmath,amssymb,amsfonts}%
\usepackage{amsthm}%
\usepackage{mathrsfs}%
\usepackage[title]{appendix}%
\usepackage{xcolor}%
\usepackage{textcomp}%
\usepackage{manyfoot}%
\usepackage{booktabs}%
\usepackage{algorithm}%
\usepackage{algorithmicx}%
\usepackage{algpseudocode}%
\usepackage{listings}%

\theoremstyle{thmstyleone}%
\newtheorem{theorem}{Theorem}[section]
\newtheorem{proposition}[theorem]{Proposition}%

\newtheorem{lemma}[theorem]{Lemma} 
\newtheorem{cor}[theorem]{Corollary} 
\newtheorem{fact}[theorem]{Fact}  

\theoremstyle{thmstyletwo}%
\newtheorem{example}[theorem]{Example}%

\theoremstyle{thmstylethree}%
\newtheorem{definition}[theorem]{Definition}%

\usepackage{hyperref} 
\usepackage{xcolor}
\usepackage[utf8]{inputenc}
\usepackage{lmodern}
\usepackage{amssymb}
\usepackage{amsmath}
\usepackage{amsthm}
\usepackage{bussproofs}
\usepackage{bussproofs-extra}
\usepackage{comment}
\usepackage{makecell}
\usepackage{leftindex}

\makeatletter
\newcommand{\leqnomode}{\tagsleft@true}
\newcommand{\reqnomode}{\tagsleft@false}
\makeatother

\usepackage{tikz}
\usetikzlibrary{arrows,calc,patterns,positioning,shapes,fit}
\usetikzlibrary{decorations.pathmorphing}
\tikzset{
modal/.style={>=stealth,shorten >=1pt,shorten <=1pt,auto,
node distance=1.5cm,semithick},
world/.style={circle,draw,minimum size=1cm,fill=gray!15},
point/.style={circle,draw,fill=black,inner sep=0.5mm},
reflexive/.style={->,in=120,out=60,loop,looseness=#1},
reflexive/.default={5},
reflexive point/.style={->,in=135,out=45,loop,looseness=#1},
reflexive point/.default={25},
coil/.style={decorate, decoration={coil,amplitude=4pt,segment length=5pt}},
snake/.style={decorate, decoration={snake}},
zigzag/.style={decorate, decoration={zigzag}}
}

\usepackage{tabularray}
\usepackage{float}

\usepackage{color, colortbl}
\definecolor{Graydark}{gray}{0.75}
\definecolor{Graylight}{gray}{0.90}

\usepackage{enumitem}

\usepackage{multirow}
\usepackage{multicol}
\usepackage{centernot}

\usepackage{enumitem}
\usepackage{physics}
\usepackage{mathrsfs}
\usepackage[mathscr]{euscript}
\usepackage{stmaryrd}
\usepackage{hyperref} 
\usepackage{oplotsymbl}

\let\DotDiamond=\trianglepbdot
\let\triangledown=\trianglepb

\newcommand{\ML}{{\mathcal{ML}}}

\begin{document}

\title[Article Title]{Axiomatizing modal inclusion logic and its variants}


\author[1]{\fnm{Aleksi} \sur{Anttila}}

\author[2]{\fnm{Matilda} \sur{Häggblom}
}

\author[3]{\fnm{Fan} \sur{Yang}}


\affil[1]{\orgdiv{Institute for Logic, Language and Computation}, \orgname{University of Amsterdam},  \country{The Netherlands}}

\affil[2]{\orgdiv{Department of Mathematics and Statistics}, \orgname{University of Helsinki}, \country{Finland}}

\affil[3]{\orgdiv{Department of Philosophy and Religious Studies}, \orgname{Utrecht University},  \country{The Netherlands}}

\abstract{We provide a complete axiomatization of modal inclusion logic---team-based modal logic extended with inclusion atoms. We review and refine an expressive completeness and normal form theorem for the logic, define a natural deduction proof system, and use the normal form to prove completeness of the axiomatization. Complete axiomatizations are also provided for two other extensions of modal logic with the same expressive power as modal inclusion logic: one augmented with a might operator and the other with a single-world variant of the might operator.}




\maketitle

\bmhead{Acknowledgments}
This research was supported by grant 336283 of Academy of Finland. The first author's research was conducted while he was affiliated with the Department of Mathematics and Statistics, University of Helsinki, Finland, where he also received funding from the European Research Council (ERC) under the European Union's Horizon 2020 research and innovation programme (grant agreement No 101020762). The second author was supported additionally by the Vilho, Yrjö and Kalle Väisälä Foundation. 

\bmhead{Corresponding author} Matilda Häggblom: matilda.haggblom@helsinki.fi.

\section{Introduction}\label{sec1}

In this article, we axiomatize modal inclusion logic and two other expressively equivalent logics. Modal inclusion logic extends the usual modal logic with \emph{inclusion atoms}---non-classical atoms the interpretation of which requires the use of \emph{team semantics}. In team semantics---introduced by Hodges \cite{hodges1997,hodges19972} to provide a compositional semantics for Hintikka and Sandu's \emph{independence-friendly logic} \cite{HinSan1989,HinBook}, and developed further by V\"{a}\"{a}n\"{a}nen in his work on \emph{dependence logic} \cite{vaananen2007}---formulas are interpreted with respect to sets of evaluation points called \emph{teams}, as opposed to single evaluation points. In Hodges' and Väänänen's first-order setting, teams are sets of variable assignments; we will mainly work in \emph{modal team semantics} (first considered in \cite{vaananen2008}), in which teams are sets of possible worlds in a Kripke model. The shift to teams enables one to express that certain relationships hold between the truth/assignment values obtained by formulas/variables in the worlds/assignments in a team, rendering team-based logics generally more expressive than their singularly evaluated counterparts. The articulation of these relationships is typically accomplished by furnishing team-based logics with various \emph{atoms of dependency} expressing different kinds of relationships. The first such atoms to be considered were Väänänen's \emph{dependence atoms} \cite{vaananen2007,vaananen2008}. Another example of note are Grädel and Väänänen's \emph{independence atoms} \cite{gradel2013}.

Inclusion atoms were introduced by Galliani in \cite{galliani2012} to import the notion of \emph{inclusion dependencies} from database theory (see, e.g., \cite{casanova1984}) into team semantics. In the modal setting, an inclusion atom is a formula of the form $\mathsf{a}\subseteq\textsf{b}$, where $\mathsf{a}$ and $\textsf{b}$ are finite sequences of formulas of modal logic of the same length. The atom $\mathsf{a}\subseteq\textsf{b}$ is satisfied in a team (i.e., a set of possible worlds) if any sequence of truth values assigned to the formulas in $\textsf{a}$ by some world in the team is also assigned to the sequence $\textsf{b}$ by some world in the team. Consider a database in which various facts about the stores in some town have been collected. Each row (world) of the database contains data about a single store, and each column (propositional symbol) corresponds to a category of collected data such as whether a store sells flowers: if a store sells
flowers, then the row for the store has a $1$ in the $\tt{Sell{\_}flowers}$ column.
If the team consisting of all rows corresponding to stores in some neighbourhood satisfies the atom $\tt{Sell{\_}flowers}\subseteq\tt{Sell{\_}mulch}$, we know that if there is a store in the neighbourhood that sells flowers, there is one that sells mulch (and that if there is one that does not sell flowers, there is one that does not sell mulch). The atom $\top\bot\subseteq\tt{Sell{\_}flowers}\,\tt{Sell{\_}food}$, on the other hand, would express that there is a store that sells flowers but not food.

Our main focus in this article is on modal logic extended with inclusion atoms, or \emph{modal inclusion logic}, which we will denote by $\ML(\subseteq)$. The expressive power of $\ML(\subseteq)$ has been studied in \cite{hella2015,kontinen2015} and its complexity in \cite{hella2019,hella20192}, but the logic has so far not been axiomatized. We fill this gap in the literature by providing a sound and complete natural deduction system for the logic. Our system is an extension of the system for propositional inclusion logic introduced in \cite{yang2022}.

In addition to $\ML(\subseteq)$, we axiomatize two other logics with the same expressive power: $\ML(\triangledown)$ and $\ML(\DotDiamond)$, or modal logic extended with a \emph{might operator} $\triangledown$ (first considered in the team semantics literature in \cite{hella2015}) and with a \emph{singular might operator} $\DotDiamond$ that we introduce in this article, respectively. The names reflect the fact that similar operators have been used to model the meanings of epistemic possibility modalities such as the ``might'' in ``It might be raining''---see, for instance, \cite{yalcin, veltman}. 

Each of these three logics---$\ML(\subseteq)$, $\ML(\triangledown)$, and $\ML(\DotDiamond)$---is \emph{closed under unions}: if a formula is true in all teams in a nonempty collection of teams, it is true in the team formed by union of the collection. The most well-known team-based logics such as dependence logic are not union closed, but union-closed logics have recently been receiving more attention in the literature (see, e.g., \cite{aloni2022,aloni2023,galliani2013,gradel2016,hella2019,hella20192,hella2015,hoelzel2021,yang2020,yang2022}). Hella and Stumpf proved in \cite{hella2015} that each of $\ML(\subseteq)$ and $\ML(\triangledown)$ is expressively complete for team properties that are closed under unions and bounded bisimulations, and which contain the empty team. We revisit this proof and also show that $\ML(\DotDiamond)$ is complete for this class of properties. These proofs of expressive completeness also yield normal forms for each of the three logics. The normal forms play a crucial role in the completeness proofs for our axiomatizations.

The structure of the article is as follows: in Section \ref{section:preliminaries}, we define the syntax and semantics of modal inclusion logic $\ML(\subseteq)$ and the might-operator logics $\ML(\triangledown)$ and $\ML(\DotDiamond)$, and discuss some basic properties of these logics. 
In Section \ref{section:expressive_power} we prove the expressive completeness results discussed above. In Section \ref{section:axiomatizations}, we introduce natural deduction systems for each of the logics and show that they are complete. We conclude the article and discuss directions for possible further research in Section \ref{section:conclusion}. We provide a translation of $\ML(\subseteq)$ into first-order inclusion logic in the \hyperref[secA1]{Appendix}.


Preliminary versions of some of the results in this article were included in the master's thesis \cite{haggblom} of the second author, which was supervised by the third and the first author.

\section{Preliminaries}\label{a} \label{section:preliminaries}

In this section, we define the syntax and semantics of 
modal inclusion logic $\ML(\subseteq)$ and the two might-operator logics $\ML(\triangledown)$ and $\ML(\DotDiamond)$. We also discuss some basic properties of these logics.



Fix a (countably infinite) set \textsf{Prop} of propositional symbols.

\begin{definition}
The syntax for the usual modal logic ($\ML$) is given by the grammar:
\begin{equation*}
   \alpha ::= p\mid\bot\mid\neg\alpha\mid(\alpha\lor\alpha)\mid(\alpha\land \alpha) \mid\Diamond\alpha\mid\Box\alpha,
\end{equation*}
where $p\in\mathsf{Prop}$. Define $\top:=\neg\bot$. We also call formulas of $\ML$ \emph{classical formulas}. Throughout the article, we reserve the first Greek letters $\alpha$ and $\beta$ for classical formulas.

The syntax for modal inclusion logic ($\ML(\subseteq)$) is given by:
\begin{equation*}
   \phi ::= p\mid\bot\mid(\alpha_1\dots\alpha_n\subseteq\beta_1\dots\beta_n)\mid\neg\alpha\mid(\phi \lor \phi) \mid(\phi \land \phi) \mid 	\Diamond \phi\mid \Box\phi,
\end{equation*}
where $p\in\mathsf{Prop}$, and $\alpha,\alpha_i,\beta_i$ (for all $1\leq i\leq n$) range over classical formulas
. We write $\mathsf{Prop}(\phi)$ for the set of propositional symbols appearing in $\phi$, and $\phi(\mathsf{X})$ if $\mathsf{Prop}(\phi)\subseteq \mathsf{X}\subseteq\mathsf{Prop}$. 
\end{definition}

The above syntax of $\ML(\subseteq)$ deserves some comment. First, we only allow negation to occur  in front of classical formulas. Next, our inclusion atoms $\alpha_1\dots\alpha_n\subseteq\beta_1\dots\beta_n$ with $\alpha_i$ and $\beta_j$ being (possibly complex) classical formulas are known in the literature also as {\em extended inclusion atoms}, and the logic $\ML(\subseteq)$ defined above is sometimes also referred to as {\em extended modal inclusion logic}---in contexts which use this terminology, a distinction is drawn between extended inclusion atoms and inclusion atoms \emph{simpliciter}, the latter referring to atoms which may only contain propositional symbols as subformulas (modal inclusion logic \emph{simpliciter}, in these contexts, is then the variant which only allows these simpler atoms). In this article we only study the extended variant. We do not allow nested inclusion atoms; for example,  $p\subseteq(p\subseteq q)$ is not a formula of $\ML(\subseteq)$. 
We usually omit the parentheses around inclusion atoms, and stipulate higher precedence for the inclusion symbol $\subseteq$ than the other connectives; for instance, $p\land q\subseteq r $ has the subformula $q\subseteq r$. 

Modal inclusion logic is interpreted on standard Kripke models, but we use teams (sets of worlds) rather than single worlds as points of evaluation. A \emph{(Kripke) model} $M=(W,R,V)$ (over $\mathsf{X}\subseteq\mathsf{Prop}$)
consists of a set $W$ of \emph{possible worlds}, a binary \emph{accessibility relation} $R\subseteq W\times W$ and a \emph{valuation function} $V:\mathsf{X}\rightarrow \mathcal{P}(W)$. 
A \emph{team} $T$ of $M$ is a subset $T\subseteq W$ of the set of worlds in $M$.
The \emph{image} of $T$ (denoted $R[T]$) and the \emph{preimage} of $T$ (denoted  $R^{-1}[T]$) are defined as
$$R[T] = \{v \in W \mid \exists w \in T :wRv\}, \text{ and}$$ 
$$R^{-1}[T] = \{w \in W \mid \exists v \in T : wRv\}.$$ 
We say that a team $S$ of $M$ is a \emph{successor team} of $T$, written $TRS$, if
$$S \subseteq R[T]\text{ and }T \subseteq R^{-1}[S];$$ 
that is, every world in $S$ is accessible 
from a world in $T$, and every world in $T$ has an accessible world in $S$.

\begin{definition}
For any Kripke model $M$ over $\mathsf{X}$ and team $T$ of $M$, the satisfaction relation $M,T\models\phi$ (or simply $T\models\phi$) for an $\ML(\subseteq)$-formula $\phi(\mathsf{X})$ is given inductively by the following clauses:
\begin{align*}
 M,T\models p  \iff &T\subseteq V(p). \\
 M,T\models\bot  \iff &T=\emptyset.\\
M,T\models\alpha_1\dots\alpha_n\subseteq\beta_1\dots\beta_n \iff& \text{for all } w\in T,\text{ there exists } v\in T \text{ such that  for all }
\\
&1\leq i\leq n,~M,\{w\}\models\alpha_i \text{ iff } M,\{v\}\models\beta_i.\\
 M,T\models\neg\alpha  \iff& M,\{w\} \not\models\alpha\text{ for all }w\in T.  \\
 M,T\models\phi\lor\psi     \iff& M,T_1\models\phi \text{ and } M,T_2\models\psi \text{ for some } T_1,T_2\subseteq T \\
 &\text{ such that }  T_1\cup T_2=T.\\
 M,T\models\phi\land\psi  \iff& M,T\models\phi \text{ and } M,T\models\psi. \\
 M,T\models\Diamond\phi          \iff&  M,S\models\phi\text{ for some } S \text{ such that } TRS.     \\
M,T\models\Box\phi          \iff&  M,R[T]\models\phi.   
 \end{align*}
We say that a set of formulas $\Gamma$ \emph{entails} $\phi$, written $\Gamma\models \phi$, if for all models $M$ and teams $T$ of $M$, if $M,T\models \gamma$ for all $\gamma\in \Gamma$, then $M,T\models \phi$. We write simply $\phi_1,\ldots,\phi_n\models\phi$ for $\{\phi_1,\ldots,\phi_n\}\models\phi$ and $\models \phi$ for $\emptyset\models \phi$, where $\emptyset$ is the empty set of formulas. If both $\phi\models\psi$ and $\psi\models\phi$, we say that $\phi$ and $\psi$ are \emph{equivalent}, and write $\phi \equiv \psi$.
\end{definition}

We often write $T\models\phi$ instead of $M,T\models\phi$.

It is easy to verify that formulas of $\ML(\subseteq)$ have the following properties:
\begin{description}
\item[\textbf{Union closure:}] If $M,T_i\models\phi$ for all $i\in I\neq \emptyset$, then $M,\bigcup_{i\in I}T_i\models\phi$.
\item[\textbf{Empty Team Property:}] $M,\emptyset\models\phi$ for all models $M$.
\end{description} 
Classical formulas $\alpha$ (i.e., formulas of $\ML$) are, in addition, \emph{downward closed}, meaning that $M,T\models\alpha$ implies $M,S\models\alpha$ for all $S\subseteq T$. It is easy to show that a formula has the downward closure property, union closure property and the empty team property if and only if it has the {\em flatness property}:
\begin{description}
\item[\textbf{Flatness:}] $M,T\models \phi$ iff $M,\{w\}\models \phi$ for every  $w\in T$.
\end{description}
Classical formulas are thus flat. It is also straightforward to verify that for any classical formula $\alpha$,
\[M,\{w\}\models \alpha 
    \iff M,w\models \alpha,\]
   where $\models$ on the right is the usual single-world-based satisfaction relation for $\ML$. It follows from this that $\ML$ (with team semantics) coincides with the usual single-world based modal logic, and hence that $\ML(\subseteq)$ is {\em conservative} over the usual modal logic.
\begin{proposition}
\label{prop:semantics-state-world}
For any set $\Gamma\cup\{\alpha\}$  of $\ML$-formulas,  
$$\Gamma\models\alpha  \iff \Gamma\models^c \alpha, $$
where $\models^c$ is the usual entailment relation for $\ML$ (over the single-world semantics).
\end{proposition} 
Given these facts, we use the notations $M,\{w\}\models \alpha$ and $M,w\models \alpha$ interchangeably whenever $\alpha$ is classical, and similarly for $\Gamma\models\alpha $ and $ \Gamma\models^c \alpha$.

Let us briefly comment on the truth conditions for the modalities. Recall our store database example from Section \ref{sec1}. If the database is also equipped with an accessibility relation detailing the (epistemically) possible future inventories of each store, we could use a modal statement such as $\Diamond(\top\bot\subseteq\tt{Sell{\_}flowers}\,\tt{Sell{\_}food})$ to express that there might in the future be a store that sells flowers and does not sell food. Note that if $S$ is a successor team of $T$, there may, for any given world $w$ in $T$, be multiple worlds accessible to $w$ in $S$. Interpretations of the modalities must take this into account---for instance, in our database example, a successor team may contain multiple inconsistent records for each store.

One may also consider alternative team-semantic truth conditions for the diamond, such as the following (called the \emph{strict semantics} for the diamond, whereas the semantics above are the \emph{lax semantics}):
$$T\models \Diamond_s\phi \iff\exists f:T\to R[T] \text{ s.t. }\forall w\in T:w R f(w)\text{ and }f[T]\models \phi,$$
where $f[T]=\{f(w)\mid w\in T\}$. With the strict semantics, $\ML(\subseteq)$ is no longer union closed, and it also fails to be bisimulation-invariant for the notion of team bisimulation established in the literature.\footnote{For the failure of union closure, consider the formula $\phi:=\Diamond_s(pq\subseteq rs)$ and a model with $W=\{w_1,w_2,w_3,w_4\}$, $R=\{(w_1,w_1),(w_2,w_2),(w_2,w_3),(w_4,w_4)\}$, and $V(p)=\{w_1,w_2\}$, $V(q)=\{w_3,w_4\}$, $V(r)=\{w_2\}$, $V(s)=\{w_3\}$. We have $\{w_1,w_2\}\models \phi$ and $\{w_2,w_4\}\models \phi$, but $\{w_1,w_2,w_4\}\not\models \phi$. For the failure of bisimulation invariance, see Section \ref{section:expressive_power}.} In this article we only consider the lax semantics. For more on strict and lax semantics, see \cite{galliani2012, yang2022, hella2019, hella20192}.\footnote{In the literature, the term ``strict semantics" typically refers to the adoption of different truth conditions not only for the diamond but also for the disjunction.} For other alternative sets of team-semantic truth conditions for the modalities, see \cite{Ciardelli2015,ciardelli2016,aloni2022, aloni2023}.

We next discuss \emph{primitive inclusion atoms}---inclusion atoms of a specific restricted form which play a crucial role in our expressive completeness results. We often abbreviate a sequence $\langle\alpha_1,\dots,\alpha_n\rangle$ of classical formulas as $\mathsf{a}$; similarly, $\mathsf{b}$ is short for $\langle\beta_1,\dots,\beta_n\rangle$, etc. We also often write $\mathsf{x}$ for a sequence $\langle x_1,\dots,x_n\rangle$ in which each $x_i$ is one of the constants $\bot$ or $\top$. The \emph{arity} of an inclusion atom $\mathsf{a}\subseteq\mathsf{b}$ is defined as the length $|\mathsf{a}|$ of the sequence $\mathsf{a}$.
 \emph{Primitive inclusion atoms} are inclusion atoms of the form $\mathsf{x}\subseteq\mathsf{a}$---that is, they are atoms with only the constants $\bot$ and $\top$ on the left-hand side. We additionally call primitive inclusion atoms $\top\dots\top\subseteq \alpha_1\dots\alpha_n$ with only the constant $\top$ on the left-hand side \emph{top inclusion atoms}. Primitive inclusion atoms $\mathsf{x}\subseteq\mathsf{a}$ are clearly \emph{upward closed} (modulo the empty team property), i.e., for any nonempty teams $T$ and $S$, if $T\models \mathsf{x}\subseteq\mathsf{a}$ and $S\supseteq T$, then $S\models \mathsf{x}\subseteq\mathsf{a} $.

Primitive inclusion atoms are relatively tractable, as the next proposition shows. Hereafter, for any $x\in\{\top,\bot\}$, we write $\alpha^\top$ for $\alpha$, and $\alpha^\bot$ for $\neg\alpha$. For $x_1,\dots,x_n\in \{\top,\bot\}$, we abbreviate $\alpha_1^{x_1}\land\dots\land\alpha_n^{x_n}$ as $\mathsf{a}^\mathsf{x}$.

 \begin{proposition}\label{prop:inclfact} For any nonempty team $T$, \[T\models\mathsf{x}\subseteq\mathsf{a}\text{ iff there exists } v\in T \text{ such that } v\models\mathsf{a}^\mathsf{x}.\]
 \end{proposition} 
 \begin{proof}
Suppose  $T\neq\emptyset$ and $T\models\mathsf{x}\subseteq\mathsf{a}$. Let $w\in T$. Then there is a $v\in T$ such that $w\models x_i$ iff $v\models\alpha_i$ for all $i=1,\dots,n$. Let $i\in\{1,\dots,n\}$. If $x_i=\top$, then $w\models x_i$, so $v\models\alpha_i$. Hence $v\models\alpha_i^\top$. If $x_i=\bot$, then $w\not\models x_i$, so $v\not\models\alpha_i$. Hence $v\models\alpha_i^\bot$. So for all $i=1,\dots,n$, $v\models\alpha_i^{x_i}$. We conclude $v\models\mathsf{a}^\mathsf{x}$. The other direction is similar. 
 \end{proof} 

\begin{cor}\label{coro:inclfact} 
$\bot\subseteq\alpha\equiv\top\subseteq\neg\alpha$. 
\end{cor}

The following example illustrates an interesting consequence of the upward closure (modulo the empty team property) of primitive inclusion atoms in our modal setting.

\begin{example}\label{example model}
Consider the model $M=(W,R,V)$ as illustrated in the figure below, where the relation $R$ is represented by the arrows, and $V(p)=\{v^\prime\}$. Consider the teams $T=\{u,v\}$ and $S=\{u',v'\}$. Since $v'\in S$ and $v'\models p$, by Proposition \ref{prop:inclfact}, $S\models \top \subseteq p$. It is also easy to verify that $TRS$, and hence $T\models\Diamond (\top\subseteq p)$. Similarly $R[T]\models \top\subseteq p$, whereby $T\models\Box(\top\subseteq p)$. In fact, we have in general $\Diamond(\mathsf{x}\subseteq \mathsf{a})\models\Box(\mathsf{x}\subseteq \mathsf{a})$ for  an arbitrary primitive inclusion atom $\mathsf{x}\subseteq \mathsf{a}$. Moreover, observe that the subteam $\{u\}$ of $T$ does not satisfy either $\Diamond(\top\subseteq p)$ or $\Box(\top\subseteq p)$, illustrating the failure of downward closure in $\ML(\subseteq)$.

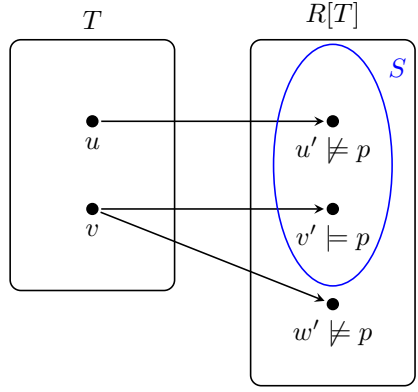
\begin{center}
\begin{tikzpicture}[modal]
\node[point] (u) [label=below:{$u$}] {};
\node[point] (v) [label=below:{$v$},below=1cm of u] {};
\node[point] (uu) [label=below:{$u^\prime\not\models p$},right=3cm of u] {};
\node[point] (vv) [label=below:{$v^\prime\models p$},right=3cm of v] {};
\node[point] (w) [label=below:{$w^\prime\not\models p$},below=1.1cm of vv] {};

\path[->] (u)  edge (uu);
\path[->] (v) edge (vv);
\path[->] (v) edge (w);

\node[draw=black, rounded corners, fit=(u) (v),inner sep=10mm,label=$T$](FIt1) {};
\node[draw=black, rounded corners, fit=(uu) (vv) (w) ,inner sep=10mm,label=$R{[}T{]}$](FIt2) {};
\node[ellipse,draw=blue, fit=(uu) (vv),inner sep=4.7mm,label=60:{\color{blue}$S$}](FIt3) {};
\end{tikzpicture}
\end{center}

\end{example}

It was proved in \cite{yang2022} in the propositional context that
arbitrary inclusion atoms can be defined in terms of primitive inclusion atoms, and further in terms of unary primitive inclusion atoms (inclusion atoms of the form $\top\subseteq\alpha$ or $\bot\subseteq \alpha$).  
This result can be extended to our current modal context by the same argument.

\begin{lemma}[\cite{yang2022}]\label{lemma:inclusion_atom_reduction} Let $\mathsf{a},\mathsf{b}$ be sequences of $\ML$-formulas, and let $\mathsf{x}, \mathsf{y}$ be sequences each of whose elements is either $\top$ or $\bot$. Then
\begin{enumerate}[label=(\roman*)]
    \item \label{lemma:inclusion_atom_reduction_i} $\bigwedge_{\mathsf{x}\in\{\top,\bot\}^{|\mathsf{a}|}}(\neg\mathsf{a}^\mathsf{x}\lor \mathsf{x}\subseteq\mathsf{b})\equiv\mathsf{a}\subseteq\mathsf{b}.$
    \item \label{lemma:inclusion_atom_reduction_ii} $\mathsf{x}\mathsf{y}\subseteq\mathsf{a}\mathsf{b}\equiv \mathsf{x}\subseteq\mathsf{a}\land((\mathsf{y}\subseteq\mathsf{b}\land\mathsf{a}^\mathsf{x})\lor\neg\mathsf{a}^\mathsf{x})$
\end{enumerate}
\end{lemma}

Given Corollary \ref{coro:inclfact}, we can further conclude that an arbitrary inclusion atom can be defined in terms of unary top inclusion atoms $\top\subseteq\alpha$. Indeed, we will see in Section \ref{section:expressive_power} that every formula of $\ML(\subseteq)$ is equivalent to one in a normal form which contains no inclusion atoms save for unary top inclusion atoms.

In addition to modal inclusion logic, we also consider two other extensions of the usual modal logic: $\ML(\triangledown)$, or modal logic with the {\em might operator} $\triangledown$, and $\ML(\DotDiamond)$, or modal logic with the {\em singular might operator} $\DotDiamond$. We define the team semantics of these operators as follows:
\begin{align*}
M,T\models\triangledown\phi&\iff T=\emptyset\text{ or there exists a nonempty } S\subseteq T\text{ such that } M,S\models\phi.\\
M,T\models\DotDiamond\phi&\iff T=\emptyset\text{ or there exists  } w\in T\text{ such that } M,\{w\}\models\phi.
\end{align*}

The operator $\triangledown$ was first considered in the team semantics literature in \cite{hella2015} and the other operator $\DotDiamond$ is introduced in the present article, but very similar operators have been employed in philosophical logic and formal semantics to model the meanings of epistemic possibility modalities such as the ``might'' in ``It might be raining'' (see, for instance \cite{yalcin, veltman}). One can also make sense of this interpretation in the team semantics setting. A team $T$ can be conceived of as corresponding to an \emph{information state}: to know that the actual world $v$ is one of the worlds in $T$ is to know that whatever holds in all worlds in $T$ must be the case. If it is raining in all worlds in $T$, this information state supports the assertion that it is raining. If, on the other hand, there is some nonempty subteam $S$ of $T$ consisting of worlds in which it is not raining, then for all that one knows given the information embodied in $T$, it might not be raining. Both operators may be thought of as expressing this kind of epistemic modality, with some subtle differences in meaning between the two (note, for instance, that if $\DotDiamond \phi$ appears within the scope of another $\DotDiamond$ in a propositional (non-modal) formula, it can be substituted \emph{salva veritate} with $\phi$, whereas the analogous fact does not hold for $\triangledown$).\footnote{\label{footnote:NE}The requirement in the truth conditions that $\triangledown \phi$ and $\DotDiamond \phi$ also be true in the empty team is added to preserve the empty team property. One may also consider variants of these operators without this requirement. The nonempty variant of $\DotDiamond$ is sometimes called the \emph{exists operator} $\mathsf{E}$ and is discussed in, for instance, \cite{kontinen2015,luck}. Logics with both the nonempty variant $\triangledown_{\normalfont\textsc{ne}}$ of $\triangledown$ and the disjunction $\vee$ may be thought of as variants of logics which include $\vee$ and the \emph{nonemptiness atom} $\normalfont\textsc{ne}$ (where $T\models \normalfont\textsc{ne} $ iff $T\neq \emptyset$) due to the equivalences $\triangledown_{\normalfont\textsc{ne}}\phi\equiv (\phi \land \normalfont\textsc{ne}) \vee \top$ and $\normalfont\textsc{ne}\equiv \triangledown_{\normalfont\textsc{ne}}\top$; for more on these logics, see \cite{yangvaananen,aloni2023}.}

The singular might $\DotDiamond\phi$  clearly entails the (non-singular) might $\triangledown\phi$; the converse direction $\triangledown\phi\models\DotDiamond\phi$ holds whenever $\phi$ is downward closed, and in particular whenever $\phi$ is classical. For classical formulas $\alpha$, both $\triangledown \alpha$ and $\DotDiamond\alpha$ coincide with the unary top inclusion atom  $\top\subseteq\alpha$.

\begin{fact}\label{equi_might}
For any $\ML$-formula $\alpha$, $\top\subseteq\alpha\equiv\;\DotDiamond\alpha\equiv\triangledown\alpha.$
\end{fact}

In general, the two might operators behave differently; in particular, there are subtle differences when the operators are iterated: $\triangledown (\triangledown \phi \land \triangledown \psi)\equiv\triangledown \phi \land \triangledown \psi$, whereas  $\DotDiamond \phi \land \DotDiamond \psi\not\models \DotDiamond (\DotDiamond \phi \land \DotDiamond \psi)$; 
$\DotDiamond( \phi \land \DotDiamond \psi) \equiv \DotDiamond (\phi \land \psi)$, whereas $\triangledown( \phi \land \triangledown \psi) \not\models \triangledown (\phi \land \psi)$.

As is the case with $\ML(\subseteq)$, the logics $\ML(\triangledown)$ and $\ML(\DotDiamond)$ are union closed and have the empty team property; we further show in Section \ref{section:expressive_power} that the three logics are in fact expressively equivalent.

As with many other logics with team semantics, the three logics we consider in this article are not closed under uniform substitution. Recall that a \emph{substitution} $\sigma$ for a logic $\mathcal{L}$ is a mapping from the set of formulas $\mathcal{L}$ to itself that commutes with the connectives (and connective-like atoms such inclusion atoms) of $\mathcal{L}$. We say that $\mathcal{L}$ is \emph{closed under} substitution $\sigma$ if for any set $\Gamma\cup\{\phi\}$ of formulas of $\mathcal{L}$,
$$\Gamma\models \phi \quad\text{implies}\quad \{\sigma(\gamma)\mid \gamma\in \Gamma\}\models \sigma(\phi).$$
A logic $\mathcal{L}$ is {\em closed under uniform substitution} if it is closed under all substitutions. Note that due to our syntactic restrictions, any map mapping a $\ML(\subseteq)$-formula $p\subseteq q$ to a non-formula $(r\subseteq s)\subseteq q$ is not considered a valid substitution for $\ML(\subseteq)$. This does not mean that substitution into inclusion atoms is disallowed in general: a valid substitution might map $p\subseteq q$ to $(r\land s)\subseteq q$.

To see why $\ML(\subseteq)$ is not closed under uniform substitution, note that clearly 
$(p\lor \neg p)\land q\models (p\land q)\lor(\neg p\land q)$ holds. But when we substitute $\top\subseteq p$ for $q$ on both sides, the entailment no longer holds: $(p\lor \neg p )\land\top\subseteq p\not\models (p\land \top \subseteq p)\lor(\neg p\land\top\subseteq p)$.
For a counterexample to the entailment, consider the team $R[T]=\{u^\prime,v^\prime,w^\prime\}$ from Example \ref{example model}. Clearly $R[T]\models(p\lor \neg p)\land \top\subseteq p$, but there are no subteams $T_1,T_2\subseteq R[T]$ such that $T_1\cup T_2=R[T]$ with $T_1\models p\land \top\subseteq p$ and $T_2\models \neg p\land \top\subseteq p$. Similar counterexamples can also be found for $\ML(\triangledown)$ and $\ML(\DotDiamond)$: in the above example, instead of using the top inclusion atom $\top\subseteq p$, we can equivalently (by Fact \ref{equi_might}) use the formula $\triangledown p$ or the formula $\DotDiamond p$.

Nevertheless, the three logics we consider in the article  are  closed under {\em classical substitutions}, namely, substitutions $\sigma$ such that $\sigma(p)$ is a classical formula for all $p\in\mathsf{Prop}$. We now prove this by using the same strategy as in \cite{yangvaananen}, where it was proved that a number of propositional team logics (including propositional inclusion logic) are closed under classical substitutions.

\begin{lemma}\label{classicalsubstlemma}
Let $\phi$ be a formula in $\ML(\subseteq)$, $\ML(\triangledown)$, or $\ML(\DotDiamond)$, and $\sigma$ a classical substitution for the relevant logic. For any model $M=(W,R,V)$ over $\mathsf{X}\supseteq \mathsf{Prop}(\sigma(\phi))$ and team $T$ of $M$, 
\[M,T \models \sigma(\phi) \iff M_\sigma,T \models \phi,\]
where $M_\sigma=(W,R,V_\sigma)$ is any model over $\mathsf{Y}\supseteq \mathsf{Prop}(\phi)$ satisfying $M_\sigma, w \models p$ iff $M,w\models\sigma(p)$, for all $p\in\mathsf{Y}$ and $w\in W$.
\end{lemma}
\begin{proof}
We prove the lemma by induction on $\phi$. The case in which $\phi=p$ follows by flatness of $\sigma(p)$. The other cases follow by the induction hypothesis. In particular, for $\phi=\Diamond\psi$, we have $M,T\models\Diamond\sigma(\psi)$ if and only if there is a team $S$ such that $TRS$ and $M,S\models\sigma(\psi)$. By the induction hypothesis, this is the case if and only if  there is a team $S$ such that $TRS$ and $M_\sigma,S\models\psi$, which holds if and only if $M_\sigma,T\models\Diamond\psi$.
\end{proof}

\begin{theorem}\label{classicalsubstthm}
Let $\Gamma\cup\{\phi\}$ be a set of formulas in $\ML(\subseteq)$, $\ML(\triangledown)$, or $\ML(\DotDiamond)$. For any classical substitution $\sigma$ for the relevant logic, if
$\Gamma\models\phi$, then $\{\sigma(\gamma) \vert \gamma\in\Gamma\}\models \sigma(\phi)$.
\end{theorem}

\begin{proof}
If $\Gamma \models \phi$, then for any model $M$ and team $T$ of $M$,
\begin{align*}
   M,T \models \sigma(\gamma)\ \text{for all } \gamma \in \Gamma &\implies M_\sigma,T \models \gamma\ \text{for all } \gamma \in \Gamma \tag{Lemma \ref{classicalsubstlemma}}\\
   &\implies M_\sigma,T \models \phi  \tag{By assumption} \\
   &\implies M,T \models \sigma(\phi) \tag{Lemma \ref{classicalsubstlemma}}
\end{align*} 
Hence $\{\sigma(\gamma) \vert \gamma \in \Gamma\} \models \sigma(\phi)$.
\end{proof}

As with standard modal logic, one can provide a first-order translation of $\ML(\subseteq)$ and its variants; see \cite{kontinen2015} for a translation (via normal form) of many modal team-based logics into classical first-order logic, and see our \hyperref[secA1]{Appendix} for a translation of $\ML(\subseteq)$ into first-order inclusion logic.


\section{Expressive completeness and normal forms} \label{section:expressive_power}

It was proved in \cite{hella2015} that each of $\ML(\subseteq)$ and $\ML(\triangledown)$ is expressively complete with respect to the class $\mathbb{U}$ of union-closed modally definable team properties with the empty team property. In this section, we review this result, and show that our new variant $\ML(\DotDiamond)$ has the same expressive power. The proofs of these results also yield a normal form for each logic. These normal forms are crucial for proving the completeness of our axiomatizations for the logics---see Section \ref{section:axiomatizations}.


As in the single-world setting, in the team-based setting modal definability can be characterized by bisimulation invariance: the modally definable properties are precisely those invariant (i.e., closed) under bisimulation. We begin by recalling the standard notion of bisimulation as well as that of Hintikka formulas---characteristic formulas which serve to capture bisimilarity in the syntax of $\ML$---and then proceed to construct analogous notions for teams.

Throughout this section, we make use of a fixed finite set $\mathsf{X}\subseteq\mathsf{Prop}$ of propositional symbols; we often omit mention of $\mathsf{X}$ in definitions and results to keep notation light. If $M$ is a Kripke model (over $\mathsf{X}$) and $w\in W$, we call $(M,w)$ a \emph{pointed model} (over $\mathsf{X}$).


\begin{definition}
For any $(M,w)$, $(M^\prime,w^\prime)$, and $k\in\mathbb{N}$, $(M,w)$ and  $(M^\prime,w^\prime)$ are $\mathsf{X},k$-bisimilar, written $M,w\leftrightarroweq^\mathsf{X}_k M^\prime,w^\prime$ (or simply $w \leftrightarroweq_k w'$), if the following recursively defined relation holds:
\begin{enumerate}[label=(\roman*)]
    \item $M,w\leftrightarroweq^\mathsf{X}_0M^\prime,w^\prime$ iff $M,w\models p \iff M^\prime,w^\prime\models p$ for all $p\in\mathsf{X}$. 
    \item $M,w\leftrightarroweq^\mathsf{X}_{k+1}M^\prime,w^\prime$ if $M,w\leftrightarroweq^\mathsf{X}_0 M^\prime,w^\prime$ and:
    \begin{enumerate}[align=left]
        \item[(Forth condition)] For every world $v$ of $M$ with $wRv$ there is a world $v^\prime$ of $M^\prime$ with $w^\prime Rv^\prime$ such that $M,v\leftrightarroweq^\mathsf{X}_{k}M^\prime,v^\prime$.
        \item[(Back condition)] For every world $v^\prime$ of $M^\prime$ with $w^\prime Rv^\prime$ there is a world $v$ of $M$ with $wRv$ such that $M,v\leftrightarroweq^\mathsf{X}_{k}M^\prime,v^\prime$.
    \end{enumerate}
\end{enumerate}
\end{definition}


The \emph{modal depth} $md(\phi)$ of a formula $\phi$ is defined by the following clauses:
\begin{align*}
md(p)&= md(\bot)=0,\\
md(\neg\alpha)&= md(\alpha), \\
md(\psi_1\lor\psi_2)&= md(\psi_1\land\psi_2)= \text{max}\{ md(\psi_1),  md(\psi_2)\}, \\
md(\Diamond\psi)&= md(\Box\psi)= md(\psi)+1, \text{ and}\\
md(\alpha_1\dots\alpha_n\subseteq\beta_1\dots\beta_n)&=\text{max}\{ md(\alpha_1),\dots md(\alpha_n), md(\beta_1)\dots md(\beta_n)\}.
\end{align*}

We say that two pointed models $(M,w)$ and $(M^\prime,w^\prime)$ are \emph{$\mathsf{X},k$-equivalent}, written $M,w \equiv^\mathsf{X}_k M^\prime,w^\prime$ (or simply $w\equiv_k w'$), if they satisfy the same $\ML$-formulas with propositional symbols among those in $\mathsf{X}$ up to modal depth $k$, i.e., if $M,w\models\alpha$ iff $M^\prime,w^\prime\models\alpha$ for every $\alpha(\mathsf{X})\in \ML$ with $md(\alpha)\leq k$.

\begin{definition} Let $k\in\mathbb{N}$ and let $(M,w)$ be a pointed model over $\mathsf{Y}\supseteq \mathsf{X}$. The $k$th \emph{Hintikka formula} $\chi^{\mathsf{X},k}_{M,w}$ (or simply $\chi^k_w$) of $(M,w)$ is defined recursively by:
\begin{enumerate}
    \item[]$\chi^{\mathsf{X},0}_{M,w}:=\bigwedge\{p\mid p\in\mathsf{X}\text{ and } w\in V(p)\}\land\bigwedge\{\neg p\mid p\in\mathsf{X}\text{ and } w\not\in V(p)\}$  
    \item[]$\chi^{\mathsf{X},k+1}_{M,w}:=\chi^k_{M,w}\land\bigwedge_{v\in R[w]}\Diamond \chi^k_{M,v}\land\Box\bigvee_{v\in R[w]}\chi^k_{M,v}$.
\end{enumerate}
\end{definition}

It is not hard to see that there are only finitely many non-equivalent $k$th Hintikka formulas for a given finite $\mathsf{X}$. This is why we may assume that the conjunction and the disjunction in $\chi^{\mathsf{X},k+1}_{M,w}$ are finite and hence that $\chi^{\mathsf{X},k+1}_{M,w}$ is well-defined\footnote{\label{footnote:finite_representatives}
To be more precise, this assumption amounts to the following: if $R[w]$ is infinite, we choose a  finite set $T\subseteq R[w]$ such that for all $v\in R[w]$, there is a $v'\in T$ with $\chi^{\mathsf{X},k}_{M,v}\equiv \chi^{\mathsf{X},k}_{M,v'}$, and define $\chi^{\mathsf{X},k+1}_{M,w}$ using this finite representative set $T$ in place of $R[w]$. Clearly, the specific choice of this finite representative set $T$ makes no difference for the Hintikka formula. Similar remarks apply to other finiteness assumptions made in the sequel. 
}.

\begin{theorem}[See, e.g., \cite{goranko}]
\label{thm:world_bisimulation}
For any $k\in\mathbb{N}$ and any pointed models $(M,w)$ and $(M^\prime,w^\prime)$:
\begin{equation*}
    w \equiv_k w^\prime \iff
    w\leftrightarroweq_{k} w^\prime
    \iff w^\prime\models\chi^k_{w}. 
\end{equation*}
\end{theorem}

We now define team-based analogues to the preceding world-based notions. A \emph{model with a team} (over $\mathsf{X}$) is a pair $(M,T)$, where $M$ is a model (over $\mathsf{X}$) and  $T$ is a team of $M$.
Team bisimulation (introduced in \cite{hella2014,kontinen2015}) is a straightforward generalization of world bisimulation:

\begin{definition}\label{teamkbisimdef}
For any $(M,T)$ and $ (M^\prime,T^\prime)$, and any $k\in\mathbb{N}$, $(M,T)$ and $( M^\prime,T^\prime)$ are (team) $\mathsf{X},k$-bisimilar, written $M,T\leftrightarroweq^\mathsf{X}_k M^\prime,T^\prime$ (or simply $T\leftrightarroweq_k T^\prime$), if: 
\begin{enumerate}[align=left]
        \item[(Forth condition)]
        For every $w\in T$ there exists a $w^\prime\in T^\prime$ such that $w\leftrightarroweq^\mathsf{X}_{k} w^\prime$.
    \item[(Back condition)] For every $w^\prime\in T^\prime$ there exists a $w\in T$ such that $w\leftrightarroweq^\mathsf{X}_{k} w^\prime$. 
\end{enumerate}
 \end{definition}

Clearly, both world and team $k$-bisimulation are equivalence relations. Moreover, 
if $w\leftrightarroweq_k w^\prime$, then $w\leftrightarroweq_n w^\prime$ for all $n\leq k$, and similarly if $T\leftrightarroweq_k T^\prime$, then $T\leftrightarroweq_n T^\prime$ for all $n\leq k$. The following lemma lists some further simple facts about team bisimulation.

\begin{lemma}[\cite{hella2014}]
\label{lemma:bisimulation_properties}
If $T\leftrightarroweq_{k+1} T^\prime$, then
\begin{enumerate}[label=(\roman*)]
\item \label{lemma:bisimulation_properties_i} For every $S$ such that $TRS$, there is a $S^\prime$ such that $T^\prime R^\prime S^\prime$ and $S\leftrightarroweq_k S^\prime$.

\item \label{lemma:bisimulation_properties_ii} $R[T]\leftrightarroweq_k R^\prime[T^\prime]$.

\item \label{lemma:bisimulation_properties_iii} For all $T_1,T_2\subseteq T$ such that $T_1\cup T_2=T$ there are $T^\prime_1,T^\prime_2\subseteq T^\prime$ such that $T^\prime_1\cup T^\prime_2=T^\prime$ and $T_i\leftrightarroweq_{k+1} T^\prime_i$ for $i\in\{1,2\}$.
\end{enumerate}
\end{lemma}

We say that the models with teams $(M,T)$ and $(M^\prime,T^\prime)$ are \emph{$\mathsf{X},k$-equivalent} (in a logic $\mathcal{L}$), written $M,T \equiv^\mathsf{X}_k M^\prime,T^\prime$ (or simply $T\equiv_k T'$), if for every formula $\phi(\mathsf{X})$ in $\mathcal{L}$ with $md(\phi)\leq k$, $M,T\models\phi$ iff $M^\prime,T^\prime\models\phi$. It will follow from results we show that $T \equiv_k T'$ in any of the logics we consider iff $T\equiv_k T'$ in either of the other logics; we therefore simply write $T\equiv_k T$ without specifying the logic. It is easy to show that $k$-bisimilarity implies $k$-equivalence for all three logics:

\begin{theorem}[Bisimulation invariance]
\label{thm:invariance}
If $T\leftrightarroweq_k T^\prime$, then $T\equiv_k T^\prime$.
\end{theorem}
 \begin{proof}
 We show by induction on the complexity of formulas $\phi$ in $\ML(\subseteq)$, $\ML(\triangledown)$, or $\ML(\DotDiamond)$ that if $T\leftrightarroweq_k T^\prime$, then $T\equiv_k T^\prime$ for $k=md(\phi)$. See \cite{hella2015} for details. In particular, items \ref{lemma:bisimulation_properties_i}, \ref{lemma:bisimulation_properties_ii} and \ref{lemma:bisimulation_properties_iii} in Lemma \ref{lemma:bisimulation_properties} are used in the diamond, box and disjunction cases, respectively. For $\phi=\DotDiamond\psi$, we have $T\models \DotDiamond\psi$ iff $T=\emptyset$ or $\{w\}\models \psi$ for some $w\in T$. If $T=\emptyset$, then by $T\leftrightarroweq_k T'$ also $T'=\emptyset$ so that $T'\models \DotDiamond\psi$. Otherwise by $T\leftrightarroweq_k T'$ we have $\{w\}\leftrightarroweq_k \{w'\}$ for some $w'\in T'$, and then $\{w'\}\models \psi$ by the induction hypothesis, whence $T'\models  \DotDiamond\psi$. The other direction is similar.
 \end{proof}

 Let us also demonstrate why $\ML(\subseteq)$, $\ML(\triangledown)$, and $\ML(\DotDiamond)$ with the strict semantics for the diamond (see Section \ref{section:preliminaries}) are not bisimulation-invariant for the notion of team bisimulation we have adopted. Consider a model $M=(W,R,V)$ over $\{p\}$ where $W=\{w_1,w_2,w_3\}$; $R=\{(w_1,w_1),(w_1,w_3),(w_2,w_2),(w_2,w_3)\}$; and $V(p)=\{w_3\}$. Clearly $w_1 \leftrightarroweq_k w_2$ for all $k\in \mathbb{N}$, so that also $\{w_1\}\leftrightarroweq_k \{w_1,w_2\}$ for all $k\in \mathbb{N}$. Defining $f(w_1):=w_3$ and $f(w_2):=w_2$, we have $f[\{w_1,w_2\}]\models \top \subseteq p\land \bot \subseteq p$ so that $\{w_1,w_2\}\models \Diamond_s (\top \subseteq p\land \bot \subseteq p)$. Clearly $\{w_1\}\not\models  \Diamond_s(\top \subseteq p\land \bot \subseteq p)$, so $\{w_1\} \leftindex_s{\not\equiv}_1 \{w_1,w_2\}$ where $\leftindex_s{\equiv}_1$ is $1$-equivalence defined with respect to the strict semantics. An analogous argument can be conducted in $\ML(\triangledown)$ as well as in $\ML(\DotDiamond)$.

With a suitable notion of bisimilarity at hand, we now proceed to show our expressive completeness results. We measure the expressive power of the logics in terms of the \emph{properties} they can express. A \emph{(team) property} (over $\mathsf{X}$) is a class of models with teams (over $\mathsf{X}$). For each formula $\phi(\mathsf{X})$ we denote by $\left\Vert\phi\right\Vert_\mathsf{X}$ (or simply $\left\Vert\phi\right\Vert$) the property (over $\mathsf{X}$) expressed or defined by $\phi$, i.e.,
\begin{equation*}
    \left\Vert\phi\right\Vert_\mathsf{X}:=\{(M,T)\text{ over }\mathsf{X}\mid M,T\models\phi\}. 
\end{equation*}
A property $\mathcal{P}$ is \emph{invariant under $\mathsf{X},k$-bisimulation} 
if $(M,T)\in\mathcal{P}$ and $M,T\leftrightarroweq^\mathsf{X}_k M^\prime,T^\prime$ imply that $(M^\prime,T^\prime)\in\mathcal{P}$; is \emph{closed under unions} 
if $(M,T_i)\in\mathcal{P}$ for all $i\in I\neq \emptyset$ implies that $(M,\bigcup_{i\in I}T_i)\in\mathcal{P}$; and has the \emph{empty team property} if $(M,\emptyset)\in\mathcal{P}$ for all $M$.

We say that a logic $\mathcal{L}$ is \emph{expressively complete} for a class of properties $\mathbb{P}$, written $\mathbb{P}=\left\Vert\mathcal{L}\right\Vert$, if for each finite $\mathsf{X}\subseteq\mathsf{Prop}$, the class $\mathbb{P}_\mathsf{X}$ of properties over $\mathsf{X}$ in $\mathbb{P}$ is precisely the class of properties over $\mathsf{X}$ definable by formulas in $\mathcal{L}$, i.e., if
$$\mathbb{P}_\mathsf{X}=\{\left\Vert\phi\right\Vert_\mathsf{X}\mid \phi \in \mathcal{L}\text{ and }\mathsf{Prop}(\phi)=\mathsf{X}\}.$$

Our goal is to show that each of the three logics is expressively complete for the class $\mathbb{U}$ of properties $\mathcal{P}$ such that $\mathcal{P}$ is union closed; $\mathcal{P}$ has the empty team property; and $\mathcal{P}$ is invariant under bounded bisimulation, meaning that if $\mathcal{P}$ is a property over $\mathsf{X}$, then $\mathcal{P}$ is invariant under $\mathsf{X},k$-bisimulation for some $k\in \mathbb{N}$. That is, we show:
$$\mathbb{U}=\left\Vert\mathcal{\ML(\subseteq)}\right\Vert=\left\Vert\mathcal{\ML(\triangledown)}\right\Vert =\left\Vert\mathcal{\ML(\DotDiamond)}\right\Vert.$$
Clearly, any property $\left\Vert\phi\right\Vert_\mathsf{X}$ expressible in any of the three logics is union closed and has the empty team property, and by Theorem \ref{thm:invariance} it is also invariant under $\mathsf{X},k$-bisimulation for some $k\in \mathbb{N}$ (take any $k\geq md(\phi)$). Therefore, $\{\left\Vert\phi\right\Vert_\mathsf{X}\mid \phi \in \mathcal{L}\text{ where }\mathsf{Prop}(\phi)=\mathsf{X}\}\subseteq\mathbb{U}_\mathsf{X}$ for $\mathcal{L}$ being any of the three logics. For the converse inclusion, we must show that any property $\mathcal{P}\in \mathbb{U}_\mathsf{X}$ is definable by some formula in each of the logics. Let us focus on $\ML(\subseteq)$ first.

As a first step, we will construct characteristic formulas in $\ML(\subseteq)$ for teams analogous to the Hintikka formulas (characteristic formulas for worlds) defined above. These formulas and the accompanying lemma are essentially as in \cite{hella2015}; we have simplified them somewhat using an idea presented in \cite{kontinen2015}.

\begin{lemma}\label{lemma:team_bisimulation}
For any $k\in\mathbb{N}$ and any $(M,T)$ over $\mathsf{X}$, there are formulas $ \eta^{\mathsf{X},k}_{M,T}\in \ML $ and $\zeta^{\mathsf{X},k}_{M,T},\theta^{\mathsf{X},k}_{M,T}\in \ML(\subseteq)$ such that:
\begin{enumerate}[label=(\roman*)]
    \item \label{lemma:team_bisimulation_i} $M^\prime,T^\prime\models\zeta^{\mathsf{X},k}_{M,T}$ iff $ T'=\emptyset$ or for all $ w\in T $ there is a $ w^\prime\in T^\prime$ such that $M,w\leftrightarroweq^\mathsf{X}_k M^\prime,w^\prime. $
\item \label{lemma:team_bisimulation_ii}
    $M^\prime, T^\prime\models \eta^{\mathsf{X},k}_{M,T}$ iff for all $w^\prime\in T^\prime $ there is a $ w\in T$ such that $M,w\leftrightarroweq^\mathsf{X}_k M^\prime, w^\prime.$
\item \label{lemma:team_bisimulation_iii} $M^\prime, T^\prime\models \theta^{\mathsf{X},k}_{M,T}$ iff $  M,T\leftrightarroweq^\mathsf{X}_k M^\prime, T^\prime\text{ or } T^\prime=\emptyset.$
\end{enumerate}
\end{lemma}
\begin{proof}
\begin{enumerate}[label=(\roman*)]
    \item Let $$\zeta^{\mathsf{X},k}_{M,T}:=\bigwedge_{w\in T}(\top\subseteq \chi^{\mathsf{X},k}_{M,w}),$$  where we stipulate $\bigwedge\emptyset=\top$. Since there are only a finite number of non-equivalent $k$th Hintikka-formulas for a given finite $\mathsf{X}$, we can assume the conjunction $\bigwedge_{w\in T}(\top\subseteq \chi^{\mathsf{X},k}_{M,w})$ to be finite and $\zeta^k_{M,T}$ to be well-defined. If $T^\prime\neq\emptyset$, then: 
    \reqnomode
\begin{align*} 
T^\prime\models\bigwedge_{w\in T}(\top\subseteq \chi^k_{w})&\iff\forall w\in T:T^\prime\models \top\subseteq \chi^k_{w} \\
&\iff\forall w\in T\ \forall v^\prime\in T^\prime\ \exists w^\prime\in T^\prime: v^\prime\models \top\iff  w^\prime\models \chi^k_{w} \\  
    &\iff \forall w\in T\ \exists w^\prime\in T^\prime: w^\prime\models \chi^k_{w} \tag{\text{Since }$T^\prime\neq\emptyset$}\\
    &\iff \forall w\in T\ \exists w^\prime\in T^\prime: w\leftrightarroweq_k w^\prime. \tag{\text{Theorem } \ref{thm:world_bisimulation}}
\end{align*}
    
    \item Let $$\eta^{\mathsf{X},k}_{M,T}:=\bigvee_{w\in T}\chi^{\mathsf{X},k}_{M,w},$$
    where we stipulate $\bigvee\emptyset=\bot$.
    \begin{align*} 
T^\prime\models\bigvee_{w\in T}\chi^k_{w}&\iff T'=\bigcup_{w\in  T}T'_w \text{ where }T'_w\models \chi^k_w \\
&\iff\forall w^\prime\in T^\prime\ \exists w\in T: w'\models \chi^k_w &&\tag{By flatness}\\
&\iff\forall w^\prime\in T^\prime\ \exists w\in T: w\leftrightarroweq_k w^\prime. &&\tag{\text{Theorem } \ref{thm:world_bisimulation}}
\end{align*}
\leqnomode

\item Let
\begin{equation*}
    \theta^{\mathsf{X},k}_{M,T}:=\eta^{\mathsf{X},k}_{M,T}\land\zeta^{\mathsf{X},k}_{M,T}.
\end{equation*}
If $T^\prime=\emptyset$, then $T'\models \theta^k_T$ by the empty team property. Else if $T^\prime\neq\emptyset$, then $T^\prime\models\eta^k_{T}\land\zeta^k_{T}$ if and only if the back and forth conditions in Definition \ref{teamkbisimdef} hold (by items \ref{lemma:team_bisimulation_i} and \ref{lemma:team_bisimulation_ii}), which is the case if and only if $T\leftrightarroweq_k T^\prime$. \qedhere
\end{enumerate}
\end{proof}

We call the characteristic formulas $\theta^k_T$ for teams obtained in Lemma \ref{lemma:team_bisimulation} \ref{lemma:team_bisimulation_iii} \emph{Hintikka formulas for teams}. As with standard Hintikka formulas, it is easy to see that for a given finite $\mathsf{X}$, there are only a finite number of non-equivalent $k$th Hintikka formulas for teams.

We are now ready to show also the inclusion 
$\mathbb{U}_\mathsf{X}\subseteq\{\left\Vert\phi\right\Vert_\mathsf{X}\mid \phi \in \mathcal{\ML(\subseteq)}$ and $\mathsf{Prop}(\phi)=\mathsf{X}\}$.
\begin{theorem}[\cite{hella2015}] \label{theorem:expressive_completeness_MIL} $\mathbb{U}=\left\Vert\mathcal{\ML(\subseteq)}\right\Vert$. That is, for each finite $\mathsf{X}\subseteq\mathsf{Prop}$, $\mathbb{U}_\mathsf{X}=\{\mathcal{C}$ over $\mathsf{X}\mid \mathcal{C}$ is union closed, has the empty team property, and is invariant under $\leftrightarroweq^\mathsf{X}_k$ for some $k\in \mathbb{N}\}=\{\left\Vert\phi\right\Vert_{\mathsf{X}}\mid\phi\in\ML(\subseteq)$ and $\mathsf{Prop}(\phi)=\mathsf{X}\}$.
\end{theorem}
\begin{proof}
The direction $\supseteq$ follows by Theorem \ref{thm:invariance} and the fact that formulas in 
$\ML(\subseteq)$ are union closed and have the empty team property.

For the direction $\subseteq$, let $\mathcal{C}\in \mathbb{U_\mathsf{X}}$, and let $k\in\mathbb{N}$ be such that $\mathcal{C}$ is invariant under $\leftrightarroweq^\mathsf{X}_k$. 
Let
\begin{equation*}
    \phi':=\bigvee_{(M,T)\in\mathcal{C}}\theta^{\mathsf{X},k}_{M,T},
\end{equation*}
where $\theta^{\mathsf{X},k}_{M,T}$ is defined as in the proof of Lemma \ref{lemma:team_bisimulation} \ref{lemma:team_bisimulation_iii}. Since there are only finitely many non-equivalent $k$th Hintikka formulas for teams for a given finite $\mathsf{X}$, we may assume the disjunction in $\phi'$ to be finite and $\phi'$ to be well-defined. We show $\mathcal{C}=\left\Vert\phi'\right\Vert_\mathsf{X}$. Note that since both $\mathcal{C}$ and $\left\Vert\phi'\right\Vert_\mathsf{X}$ have the empty team property, neither is empty.

Let $(M,T)\in\mathcal{C}$. Clearly $T\leftrightarroweq_k T$, so by Lemma \ref{lemma:team_bisimulation} \ref{lemma:team_bisimulation_iii}, $T\models\theta^k_{T}$, whence $T\models\phi^\prime$.

Now let $M^\prime,T^\prime\models\phi^\prime$. Then there are subteams $T_{T}^\prime\subseteq T^\prime$ such that $T^\prime=\bigcup_{(M,T)\in \mathcal{C}}T_{T}^\prime$ and $M^\prime,T_{T}^\prime\models\theta^k_{T}$. By Lemma \ref{lemma:team_bisimulation} \ref{lemma:team_bisimulation_iii} it follows that for a given $(M,T)\in \mathcal{C}$, either $M,T\leftrightarroweq_k M^\prime, T_{T}^\prime$ or $T_{T}^\prime=\emptyset$. If $T_{T}^\prime=\emptyset$, then by the empty team property $(M^\prime, T_{T}^\prime)\in \mathcal{C}$. If $M,T\leftrightarroweq_k M^\prime, T_{T}^\prime$, then since $\mathcal{C}$ is invariant under $k$-bisimulation, $(M^\prime, T_{T}^\prime)\in\mathcal{C}$. So for all $(M,T)\in \mathcal{C}$ we have $(M^\prime, T_{T}^\prime)\in\mathcal{C}$; then since $\mathcal{C}$ is closed under unions, we conclude that $(M^\prime, T^\prime)\in\mathcal{C}$. 
\end{proof}

It follows from the proof of Theorem \ref{theorem:expressive_completeness_MIL} that each $\ML(\subseteq)$-formula is equivalent to a formula in the normal form \begin{equation*}\label{NF} \tag*{(NF)}
\bigvee_{(M,T)\in\mathcal{C}}\theta^k_{T}=    \bigvee_{(M,T)\in\mathcal{C}}(\bigvee_{w\in T}\chi^k_{w}\land\bigwedge_{w\in T}(\top\subseteq\chi^k_{w})).   
\end{equation*}
Note that in this normal form, the extended inclusion atoms $\top \subseteq \chi^k_{w}$ play a substantial role that cannot be replicated using inclusion atoms \emph{simpliciter}, i.e., inclusion atoms $p_1\dots p_n\subseteq q_1\dots q_n$ with propositional symbols only. It was observed in \cite{yang2022} that the variant of propositional inclusion logic with only these simpler inclusion atoms is strictly less expressive than extended propositional inclusion logic, as, e.g., in one propositional symbol $p$, the former contains only one (trivial) inclusion atom $p\subseteq p$ and is thus flat, while the latter can express non-flat properties using extended inclusion atoms such as $\top\subseteq p$ and $\bot \subseteq p$. The analogous fact clearly also holds in the modal case.

Observe also that only unary top inclusion atoms (atoms of the form $\top \subseteq \alpha$) are required in the above normal form to express any given property in $\mathbb{U}$. Given that $\top \subseteq \alpha\equiv  \triangledown \alpha \equiv \DotDiamond \alpha$ (Fact \ref{equi_might}), we can derive the expressive completeness of $ \ML(\triangledown)$ as well as that of $\ML(\DotDiamond)$ as a corollary to Theorem \ref{theorem:expressive_completeness_MIL}.
\begin{theorem} \label{theorem:expressive_completeness_might}$\mathbb{U}=\left\Vert\ML(\triangledown)\right\Vert=\left\Vert\ML(\DotDiamond)\right\Vert$.
\end{theorem}
And we clearly have the following normal forms for these logics:
\begin{align*}\tag*{($\triangledown$NF)}
    \bigvee_{(M,T)\in\mathcal{C}}(\bigvee_{w\in T}\chi^k_{w}\land\bigwedge_{w\in T}\triangledown\chi^k_{w}).\\
    \tag*{($\DotDiamond$NF)}\bigvee_{(M,T)\in\mathcal{C}}(\bigvee_{w\in T}\chi^k_{w}\land\bigwedge_{w\in T}\DotDiamond\chi^k_{w}).   
\end{align*}
The expressive completeness of $\ML(\triangledown)$ was also proved in \cite{hella2015}; we have added the result for $\ML(\DotDiamond)$. Let us briefly comment on how the structure of the normal form motivates the introduction of the operator $\DotDiamond$. Given a team $T$, the formula $\zeta^k_T=\bigwedge_{w\in T}(\top\subseteq \chi^{k}_{w})$ is intended to express the forth-condition of $\leftrightarroweq_k$ from $T$ to a team $T'$ in the sense that $T'\models \zeta^k_T$ iff for all $w\in T$ there is a $w'\in T$ such that $w\leftrightarroweq_k w'$ (see the proof of Lemma \ref{lemma:team_bisimulation} \ref{lemma:team_bisimulation_i}). The formula $\bigwedge_{w\in T}\DotDiamond \chi^k_w$ articulates precisely what $T'$ needs to satisfy in order to fulfill this condition. Inclusion atoms are clearly stronger than is strictly necessary if we only wish to express this condition. The operator $\triangledown$ is also \emph{prima facie} stronger than $\DotDiamond$ in the sense that for any $\phi\in \ML(\DotDiamond)$, there is an $\alpha \in \ML$ such that $\DotDiamond \phi \equiv \DotDiamond \alpha$,\footnote{This follows from the fact that $\ML$ is expressively complete for the class of flat properties invariant under $k$-bisimulation for some $k\in \mathbb{N}$ (see, e.g., \cite{yang2017}). A team property $\mathcal{P}$ is \emph{flat} if $(M,T)\in \mathcal{P} \iff (M,\{w\})\in \mathcal{P}$ for all $w\in T$. Given any $\phi\in \ML(\DotDiamond)$, $\{(M,T)\mid M,\{w\}\models \phi$ for all $w\in T\}$ is clearly flat 
and invariant under $k$-bisimulation for $k=md(\phi)$, so by the expressive completeness fact we have $\{(M,T)\mid M,\{w\}\models \phi$ for all $w\in T\}=\left\Vert\alpha \right\Vert$ for some $\alpha\in \ML$, and then  $\DotDiamond\phi \equiv\DotDiamond\alpha $.}
 whereas, for instance, there is no $\alpha\in \ML $ such that $\triangledown (\triangledown p \land \triangledown q)\equiv \triangledown \alpha$---the singular might, in a sense, cancels out non-classical content within its scope. It is interesting, therefore, that $\ML(\DotDiamond)$ has the same expressive power as the other two logics. 

One important corollary to the results in this section is that the logics we consider are compact. This follows from the fact that any team property invariant under bounded bisimulation can be expressed in classical first-order logic---see \cite{kontinen2015}. (It should be noted that for many team-based logics, compactness formulated in terms of satisfiability need not coincide with compactness formulated in terms of entailment, but the logics we consider are compact in both senses; see the \hyperref[secA1]{Appendix} for further discussion.)
It also follows readily from the expressive completeness results that 
each of the three logics admits uniform interpolation. It was proved in \cite{dagostino2019} that any team-based 
modal logic which is \emph{local} and \emph{forgetting} admits uniform interpolation. The logics we consider are all local, and it follows from the expressive completeness theorems that they are also forgetting; they therefore admit uniform interpolation. We refer the reader to \cite{dagostino2019} for the detailed proof and more discussion about the notion of uniform interpolation; see \cite{yang2022} for clarification on the role of locality.

\section{Axiomatizations}\label{b} \label{section:axiomatizations}
In this section, we introduce sound and complete natural deduction systems for the logics $\ML(\subseteq)$, $\ML(\triangledown)$ and $\mathcal{ML}(\DotDiamond)$. We prove the completeness of the axiomatizations by means of a strategy commonly used for propositional and modal team-based logics which involves showing that each formula is provably equivalent to its normal form. This strategy is used in \cite{yang2022} for the propositional inclusion logic system; we also adapt many of the lemmas and other details of the proof presented in \cite{yang2022}.

\subsection{$\ML(\subseteq)$}\label{t}

We start by axiomatizing our core logic, modal inclusion logic $\ML(\subseteq)$. Our natural deduction system for $\ML(\subseteq)$ comprises standard rules for connectives and modalities for the smallest normal modal logic $\mathcal{K}$, modified to account for special features of our setting such as the failure of downward closure, as well as rules governing the behaviour of inclusion atoms and their interaction with the other connectives. Note that since $\ML(\subseteq)$ is not closed under uniform substitution (see Section \ref{section:preliminaries}), the system does not admit the usual uniform substitution rule.

\begin{definition}
The natural deduction system for $\ML(\subseteq)$ consists of all axioms and rules presented in Tables \ref{table:classical_connectives}-\ref{table:modal_inclusion}. We write $\Gamma\vdash_{\ML(\subseteq)}\phi$ (or simply $\Gamma\vdash \phi$) if $\phi$ is derivable from formulas in $\Gamma$ using this system. We write $\phi\vdash \psi$ for $\{\phi\}\vdash \psi$ and $\vdash \phi$ for $\emptyset\vdash \phi$, and say that $\phi$ and $\psi$ are \emph{provably equivalent}, denoted by $\phi\dashv\vdash\psi$, if $\phi\vdash\psi$ and $\psi\vdash\phi$.
\end{definition}

\begin{table}[h]\centering \small
  \renewcommand*{\arraystretch}{3.8}
\begin{tabular*}{\linewidth}{@{\extracolsep{\fill}}|ccc|}
\hline
     \AxiomC{$[\alpha]$}
     \noLine
     \UnaryInfC{$D_0$} 
     \noLine     
     \UnaryInfC{$\bot$}
     \RightLabel{$\neg$I\scriptsize(1)}
\UnaryInfC{$\neg\alpha$}
\DisplayProof
&
  \AxiomC{$[\neg\alpha]$} 
  \noLine
  \UnaryInfC{$D_0$}
  \noLine
  \UnaryInfC{$\bot$}
  \RightLabel{RAA\scriptsize(1)}
 \UnaryInfC{$\alpha$}  
 \DisplayProof
& \AxiomC{$D_0$}
 \noLine
 \UnaryInfC{$\alpha$}
  \AxiomC{$D_1$}
 \noLine
 \UnaryInfC{$\neg\alpha$} \RightLabel{$\neg$E}
 \BinaryInfC{$\phi$}
 \DisplayProof
 \\ [2ex] 
\AxiomC{$D$}
\noLine
\UnaryInfC{$\phi$}
\RightLabel{$\lor$I}
\UnaryInfC{$\phi\lor\psi$}
\DisplayProof
& 
\AxiomC{$D$}
\noLine
\UnaryInfC{$\psi$}
\RightLabel{$\lor$I}
\UnaryInfC{$\phi\lor\psi$}
\DisplayProof &  \AxiomC{$D$}
\noLine
\UnaryInfC{$\phi \lor \psi$}
\noLine
\AxiomC{[$\phi$]}
\noLine
\UnaryInfC{$D_0$}
\noLine
\UnaryInfC{$\chi$}
\AxiomC{[$\psi$]}
\noLine
\UnaryInfC{$D_1$}
\noLine
\UnaryInfC{$\chi$}
\RightLabel{$\lor$E\scriptsize(1)}
\TrinaryInfC{$\chi$}
\DisplayProof\\ [2ex] 
\multicolumn{2}{|c}{
\AxiomC{$D_0$}
\noLine
\UnaryInfC{$\phi$}
\AxiomC{$D_1$}
\noLine
\UnaryInfC{$\psi$}
\RightLabel{$\land$I}
\BinaryInfC{$\phi\land\psi$}
\DisplayProof}
&
\AxiomC{$D$}
\noLine
\UnaryInfC{$\phi \land \psi$}
\RightLabel{$\land$E}
\UnaryInfC{$\phi$}
\DisplayProof
\hskip 1.5em
\AxiomC{$D$}
\noLine
\UnaryInfC{$\phi \land \psi$}
\RightLabel{$\land$E}
\UnaryInfC{$\psi$}
\DisplayProof \\ 

\AxiomC{$D$}
\noLine
\UnaryInfC{$\neg\Box\alpha$}
\RightLabel{$\Diamond\Box$Inter}
\UnaryInfC{$\Diamond\neg\alpha$}
\DisplayProof
&

\AxiomC{$D$}
\noLine
\UnaryInfC{$\Diamond\neg\alpha$}
\RightLabel{$\Box\Diamond$Inter}
\UnaryInfC{$\neg\Box\alpha$}
\DisplayProof 
&

\AxiomC{$D$}
\noLine
\UnaryInfC{$\Diamond(\phi\lor\psi)$}
\RightLabel{$\Diamond\lor$Distr}
\UnaryInfC{$\Diamond\phi\lor\Diamond\psi$}
\DisplayProof \\[5ex]

\multicolumn{3}{|l|}{\kern-3em\AxiomC{$[\phi_1]$}
\AxiomC{$\dots$}
\AxiomC{$[\phi_n]$}
\noLine
\TrinaryInfC{$D_0$\ \ }
\noLine
\UnaryInfC{$\ddots\ $\vdots$\ \reflectbox{$\ddots$}\ \ $}
\noLine
\UnaryInfC{$\psi\ \ $}

\AxiomC{$D_1$}
\noLine
\UnaryInfC{$\Box\phi_1$}

\AxiomC{$\dots$}
\noLine
\UnaryInfC{}
\noLine
\UnaryInfC{}
\noLine
\UnaryInfC{$\dots$}

\AxiomC{$D_n$}
\noLine
\UnaryInfC{$\Box\phi_n$}

\RightLabel{$\Box$Mon\scriptsize(2)}
\QuaternaryInfC{$\Box\psi$}
\DisplayProof

\kern-1em\AxiomC{$[\phi]$}
\noLine
\UnaryInfC{$D_0$}
\noLine
\UnaryInfC{$\psi$}
\AxiomC{$D_1$}
\noLine
\UnaryInfC{$\Diamond\phi$}
\RightLabel{$\Diamond$Mon\scriptsize(2)}
\BinaryInfC{$\Diamond\psi$}
\DisplayProof}   \\
 \multicolumn{3}{|l|}{\makecell{(1) The undischarged assumptions in $D_0$ and $D_1$ are $\ML$-formulas. \\
 (2) $D_0$ has no undischarged assumptions.\,\,\,\,\,\,\,\,\,\,\,\,\,\,\,\,\,\,\,\,\,\,\,\,\,\,\,\,\,\,\,\,\,\,\,\,\,\,\,\,\,\,\,\,\,\,\,\,\,\,\,\,\,\,\,\,\,\,\,\,\,\,\,\,\,\,\,\,\,}}\\
 \hline
\end{tabular*}
 \caption{Rules for classical modal logic.}
\label{table:classical_connectives}
\end{table}  
 
\begin{table}[h]\centering \small
  \renewcommand*{\arraystretch}{3.8}
\begin{tabular*}{\linewidth}{@{\extracolsep{\fill}}| c c |}
       \hline
\AxiomC{}
\RightLabel{$\subseteq$Id}
\UnaryInfC{$\mathsf{a}\subseteq\mathsf{a}$}
\DisplayProof  &
\AxiomC{$D_0$}
\noLine
\UnaryInfC{$\alpha^x$}
\AxiomC{$D_1$}
\noLine
\UnaryInfC{$\mathsf{b}\subseteq\mathsf{c}$}
\RightLabel{$\subseteq$Exp}
\BinaryInfC{$x\mathsf{b}\subseteq\alpha\mathsf{c}$}
\DisplayProof \\ [2ex] 
 \AxiomC{$D_0$}
\noLine
\UnaryInfC{$\neg\mathsf{a}^\mathsf{x}$}
\AxiomC{$D_1$}
\noLine
\UnaryInfC{$\mathsf{x}\subseteq\mathsf{a}$}
\RightLabel{$\subseteq_\neg$E}
\BinaryInfC{$\phi$}
\DisplayProof
& {\AxiomC{$D$}
\noLine
\UnaryInfC{$\mathsf{x}\subseteq\mathsf{a}\lor\psi$}
\AxiomC{$[\mathsf{x}\subseteq\mathsf{a}]$}
\noLine
\UnaryInfC{$D_0$}
\noLine
\UnaryInfC{$\chi$}
\AxiomC{$[\psi]$}
\noLine
\UnaryInfC{$D_1$}
\noLine
\UnaryInfC{$\chi$}
\RightLabel{$\lor_\subseteq$E}
\TrinaryInfC{$\chi$}
\DisplayProof}
\\
  \AxiomC{$D$}
    \noLine
    \UnaryInfC{$\bigwedge_{\mathsf{x}\in\{\top,\bot\}^{|\mathsf{a}|}}(\neg\mathsf{a}^\mathsf{x}\lor \mathsf{x}\subseteq\mathsf{b})$}
    \RightLabel{$\subseteq$Ext}
    \UnaryInfC{$\mathsf{a}\subseteq\mathsf{b}$}
    \DisplayProof &   \AxiomC{$D$}
    \noLine
    \UnaryInfC{$\mathsf{a}\subseteq\mathsf{b}$}
    \RightLabel{$\subseteq$Rdt}    
    \UnaryInfC{$\neg\mathsf{a}^\mathsf{x}\lor \mathsf{x}\subseteq\mathsf{b}$}
    \DisplayProof  
 \\ 
      \multicolumn{2}{|c|}{\AxiomC{$D_0$}\kern-2em
\noLine
    \UnaryInfC{$\phi\lor\psi$}
\AxiomC{$D_1$}
\noLine    
    \UnaryInfC{$ \mathsf{x}_1\subseteq \mathsf{a}_1$}
\AxiomC{$\dots$}
\noLine  
\UnaryInfC{}
\noLine  
\UnaryInfC{}
\noLine  
\UnaryInfC{$\dots$}
\AxiomC{$D_n$}
\noLine    
    \UnaryInfC{$ \mathsf{x}_n\subseteq \mathsf{a}_n$}
\RightLabel{$\subseteq$Distr}
    \QuaternaryInfC{$((\phi\lor \mathsf{a}_1^{\mathsf{x}_1}\lor\dots\lor\mathsf{a}_n^{\mathsf{x}_n})\land  \mathsf{x}_1\subseteq \mathsf{a}_1\land\dots\land \mathsf{x}_n\subseteq \mathsf{a}_n)\lor\psi$}
    \DisplayProof}\\  [2ex]
    \hline
\end{tabular*}
 \caption{Rules for inclusion.}
\label{table:inclusion_rules}
\end{table}

\begin{table}[h]\centering \small
  \renewcommand*{\arraystretch}{3.8}
\begin{tabular*}{\linewidth}{@{\extracolsep{\fill}} |c c |}
\hline
\AxiomC{$D$}
\noLine
\UnaryInfC{$\Diamond(\mathsf{x}\subseteq\mathsf{a})$}
\RightLabel{$\Diamond_\subseteq$Distr}
\UnaryInfC{$\top\subseteq\Diamond\mathsf{a}^{\mathsf{x}}$}
\DisplayProof 
& 
\AxiomC{$D$}
\noLine
\UnaryInfC{$\top\subseteq\Diamond\mathsf{a}^\mathsf{x}$}
\RightLabel{$\Diamond\Box_\subseteq$Exc}
\UnaryInfC{$\Box(\mathsf{x}\subseteq\mathsf{a})$}
\DisplayProof \\[2ex]
 \AxiomC{$D_0$}
    \noLine
    \UnaryInfC{$\top\subseteq\Diamond\beta$}
    \AxiomC{$D_1$}
    \noLine
    \UnaryInfC{$\Box(\mathsf{x}\subseteq\mathsf{a})$}
    \RightLabel{$\Box\Diamond_\subseteq$Exc}
    \BinaryInfC{$\top\subseteq\Diamond\mathsf{a}^\mathsf{x}$}
    \DisplayProof & 
\kern-2em\AxiomC{$D$}
\noLine
\UnaryInfC{$\Box(\mathsf{x}\subseteq\mathsf{a}\lor\psi)$}
\AxiomC{$[\top\subseteq \Diamond\mathsf{a}^\mathsf{x}]$}
\noLine
\UnaryInfC{$D_0$}
\noLine
\UnaryInfC{$\chi$}
\AxiomC{$[\Box\psi]$}
\noLine
\UnaryInfC{$D_1$}
\noLine
\UnaryInfC{$\chi$}
\RightLabel{$\Box\lor_\subseteq$E}
\TrinaryInfC{$\chi$}
\DisplayProof \\
\multicolumn{2}{|c|}
{\kern-2em\AxiomC{$D_0$}
\noLine
\UnaryInfC{$\Diamond\phi$}
\AxiomC{$D_1$}
\noLine
\UnaryInfC{$x_1\subseteq\Diamond\alpha_1$}
\AxiomC{$\dots$}
\noLine  
\UnaryInfC{}
\noLine  
\UnaryInfC{}
\noLine  
\UnaryInfC{$\dots$}
\AxiomC{$D_n$}
\noLine
\UnaryInfC{$x_n\subseteq\Diamond\alpha_n$}
\RightLabel{$\subseteq_\Diamond$Distr}
\QuaternaryInfC{$\Diamond((\phi\lor\alpha_1^{x_1}\vee\ldots \vee\alpha_n^{x_n})\land x_1\subseteq\alpha_1\land \ldots x_n\subseteq \alpha_n)$}
\DisplayProof}  \\ [2ex]
\hline
\end{tabular*}
 \caption{Rules for modal operators and inclusion.}
\label{table:modal_inclusion} 
\end{table}
Table \ref{table:classical_connectives} lists the rules which operate only on the connectives and operators of $\ML$. 
The soundness of the disjunction elimination rule as well as that of the negation rules involving assumptions requires that the undischarged assumptions in the subderivations are downward closed, hence the restriction to $\ML$-formulas. 

Table \ref{table:inclusion_rules} lists the propositional rules for inclusion atoms. All rules in this table are adapted from rules or results in the propositional inclusion logic system in \cite{yang2022}, but we have considerably simplified the propositional system (see Proposition \ref{prop:derivability_results}, in which we show that the full set of rules from \cite{yang2022} is derivable in our system). The rule $\subseteq_\neg$E captures the fact that $\mathsf{x}\subseteq\mathsf{a}$ and $\neg\mathsf{a}^\mathsf{x}$ lead to a contradiction, since (for a nonempty team) $\mathsf{x}\subseteq\mathsf{a}$ implies that $\mathsf{a}^\mathsf{x}$ is true somewhere in the team (see Proposition \ref{prop:inclfact}), while $\neg\mathsf{a}^\mathsf{x}$ implies that $\mathsf{a}^\mathsf{x}$ is not true anywhere in the team. The rules $\subseteq$Ext and $\subseteq$Rdt allow us to reduce arbitrary inclusion atoms to equivalent formulas in which all non-classical subformulas are primitive inclusion atoms (see Lemma \ref{lemma:inclusion_atom_reduction}) and to utilize the properties of classical formulas to derive properties of inclusion atoms (see Proposition \ref{prop:derivability_results}). The rule $\lor_\subseteq$E models the upward closure (modulo the empty team property) of primitive inclusion atoms: if $T\models \mathsf{x}\subseteq \mathsf{a}\vee\psi$ and $T\not \models \psi$, then $T=S\cup U$ where $S\models \mathsf{x}\subseteq \mathsf{a}$, $U\models \psi$ and $S\neq \emptyset$, whence also $T\models \mathsf{x}\subseteq \mathsf{a}$. Primitive inclusion atoms are not downward closed so we do not in general have the converse direction: $(\phi\lor\psi)\land \mathsf{x}\subseteq\mathsf{a}\not\models(\phi\land\mathsf{x}\subseteq\mathsf{a})\lor\psi$. However, $(\phi\lor\psi)\land \mathsf{x}\subseteq\mathsf{a}\models((\phi\lor\mathsf{a}^\mathsf{x})\land\mathsf{x}\subseteq\mathsf{a})\lor\psi$ does hold, as reflected in the rule $\subseteq$Distr.

The rules in Table \ref{table:modal_inclusion} concern modal operators and inclusion atoms.
The rule $\Diamond_\subseteq$Distr allows us to distribute the diamond over the inclusion atom. The converse direction of this rule is not sound, as, e.g., $\top\subseteq\Diamond p \not\models\Diamond(\top\subseteq p)$ (consider a team with a world $w\models \Diamond p$ but without a successor team). If we ensure that there is a successor team by requiring $\Diamond \phi$ to hold for some $\phi$, the converse is sound. The rule $\subseteq_\Diamond$Distr generalizes this fact, and indeed we have:
\begin{align*}
\Diamond\phi,\top\subseteq\Diamond\mathsf{a}^\mathsf{x}&\vdash\Diamond((\phi\lor(\mathsf{a}^\mathsf{x})^\top)\land\top\subseteq\mathsf{a}^\mathsf{x})\tag{$\subseteq_\Diamond $Distr}\\
&\vdash\Diamond(\top\subseteq\mathsf{a}^\mathsf{x}) \tag{$\Diamond $Mon}   
\end{align*}
The rule $\Diamond\Box_\subseteq$Exc allows us to derive a box formula from a top inclusion formula with a diamond formula on the right. The converse is not sound---consider a nonempty team $T$ with $R[T]=\emptyset$. By the empty team property, $T\models\Box(\top\subseteq\alpha)$, whereas  $T\not\models\top\subseteq\Diamond\alpha$, since the worlds in $T$ have no accessible worlds. We can ensure that $R[T]\neq \emptyset$ in case $T\neq \emptyset$ by requiring $\top\subseteq\Diamond\beta$ to hold for some $\beta$---this yields the rule $\Box\Diamond_\subseteq$Exc. The rule $\Box\lor_\subseteq$E is similar to $\lor_\subseteq$E in Table \ref{table:inclusion_rules}, but it applies to box formulas.

\begin{theorem}[Soundness]\label{soundness}
If $\Gamma\vdash\phi$, then $\Gamma\models\phi$.
\end{theorem}
\begin{proof}
Soundness proofs for most rules in Tables \ref{table:classical_connectives} and \ref{table:inclusion_rules} can be found in \cite{yang2022}. 
The soundness of $\Diamond_\subseteq$Distr, $\Box\Diamond_\subseteq$Exc and $\Diamond\Box_\subseteq$Exc (Table \ref{table:modal_inclusion}) is easy to show using Proposition \ref{prop:inclfact}. We only give detailed proofs for $\subseteq_\Diamond$Distr and $\Box\lor_\subseteq$E (Table \ref{table:modal_inclusion}).
\begin{enumerate}[align=left]
    
    \item[($\subseteq_\Diamond$Distr)] It suffices to show that $\Diamond\phi\land\bigwedge_{1\leq i\leq n}(x_i\subseteq\Diamond\alpha_i)\models\Diamond((\phi\lor\bigvee_{1\leq i\leq n}\alpha_i^{x_i})\land\bigwedge_{1\leq i\leq n}(x_i\subseteq\alpha_i))$. Suppose that $T\models\Diamond\phi\land\bigwedge_{i\in I}(x_i\subseteq\Diamond\alpha_i)$. If $T=\emptyset$, it satisfies the conclusion by the empty team property so we may suppose $T\neq \emptyset$. Then there is some $S$ such that $TRS$ and $S\models\phi$, and for each $1\leq i \leq n$, $T\models x_i\subseteq\Diamond\alpha_i$. If $x_i=\top$, then there is some $v_i\in R[T]$ such that $v_i\models\alpha_{i}$, i.e., $v_i\models\alpha_{i}^{x_i}$. If, on the other hand, $x_i=\bot$, then by Proposition \ref{prop:inclfact} there is some $w_i\in T$ such that $w_i\models\neg\Diamond\alpha_i$, i.e., $w_i\models\Box\neg\alpha_i$. By $TRS$ there is a world $v_i\in R[T]$ such that $w_i R v_i$, and since $w_i\models\Box\neg\alpha_i$, we have that $v_i\models\neg\alpha_i$, i.e., $v_i\models\alpha_{i}^{x_i}$. Let $S^\prime:=S\cup\{v_i\}_{1\leq i\leq n}.$ Clearly $S^\prime\models(\phi\lor\bigvee_{1\leq i\leq n}\alpha_i^{x_i})\land\bigwedge_{1\leq i\leq n}(x_i\subseteq\alpha_i)$. Since $S\subseteq S^\prime\subseteq R[T]$ and $TRS$, we have that $TRS^\prime$, so $T\models\Diamond((\phi\lor\bigvee_{1\leq i\leq n}\alpha_i^{x_i})\land\bigwedge_{1\leq i\leq n}(x_i\subseteq\alpha_i))$.
    
    \item[($\Box\lor_\subseteq$E)] Suppose that $\Gamma, \top\subseteq\Diamond\mathsf{a}^\mathsf{x}\models\chi$, and $\Gamma,\Box\psi\models\chi$. Let $T\models\Box(\mathsf{x}\subseteq\mathsf{a}\lor\psi)$ and $T\models\gamma$ for all $\gamma\in\Gamma$.  We show that $T\models\chi$.
    
    We have that $R[T]\models\mathsf{x}\subseteq\mathsf{a}\lor\psi$, so $R[T]=T_1\cup T_2$ where $T_1\models \mathsf{x}\subseteq\mathsf{a}$ and $T_2\models\psi$. If $T_1=\emptyset$, then $T_2=R[T]$, so $R[T]\models\psi$, whence $T\models\Box\psi$; therefore also $T\models\chi$.
Else if $T_1\neq \emptyset$, then by $T_1\models \mathsf{x}\subseteq\mathsf{a}$ and Proposition \ref{prop:inclfact} it follows that there is some $v\in T_1$ such that $v\models\mathsf{a}^\mathsf{x}$. By $v\in R[T]$, there is a $w\in T$ with $wRv$, whence $w\models\Diamond\mathsf{a}^\mathsf{x}$. Therefore  $T\models \top\subseteq\Diamond\mathsf{a}^\mathsf{x}$, whence also $T\models\chi$.\qedhere
\end{enumerate}
\end{proof}

It is easy to verify that the rules in Table \ref{table:classical_connectives} constitute a system that is equivalent to other axiomatizations of $\mathcal{K}$,\footnote{For a proof that a similar system is equivalent to a Hilbert-style system for $\mathcal{K}$, see \cite{yang2017}. Note that some of the rules in Table \ref{table:classical_connectives} such as $\Diamond\lor$Distr are derivable for the classical fragment of our axiomatization and are therefore not necessary for this equivalence.} and so, given soundness, Proposition \ref{prop:semantics-state-world}, and the fact that $\mathcal{K}$ is complete for the class of all Kripke models, we have:
\begin{proposition}[Classical completeness] \label{prop:classical_completeness}
Let $\Gamma\cup\{\alpha\}$ be a set of $\ML$-formulas. Then
 $$\Gamma\models\alpha\iff\Gamma\vdash_{\ML(\subseteq)}\alpha.$$
\end{proposition}

In the following proposition 
we derive some useful properties of inclusion atoms. These properties allow us to manipulate the atoms with ease, and we occasionally make use of them without explicit reference to this proposition. All items in the proposition were included as rules in the propositional system in \cite{yang2022}; the proposition shows that the simpler propositional system yielded by the propositional fragment of the rules in Tables \ref{table:classical_connectives} and \ref{table:inclusion_rules} suffices.\footnote{The system in \cite{yang2022} did not feature our $\subseteq_\lnot$E as a rule; it was instead shown to be derivable in that system using a rule corresponding to our Proposition \ref{prop:derivability_results} \ref{prop:derivabilility_results_monotonicity}. An alternative simplified propositional system would then replace $\subseteq_\lnot$E with Proposition \ref{prop:derivability_results} \ref{prop:derivabilility_results_monotonicity}. The rule corresponding to Proposition \ref{prop:derivability_results} \ref{prop:derivabilility_results_monotonicity} was originally introduced in the first-order system in \cite{hannula}.  
}

\begin{proposition}\label{prop:derivability_results}\
\begin{enumerate}[label=(\roman*)]

 \item \label{prop:derivability_results_transitivity}
 $\mathsf{a}\subseteq\mathsf{b},\mathsf{b}\subseteq\mathsf{c}\vdash\mathsf{a}\subseteq\mathsf{c}.$
 
 \item \label{prop:derivability_results_commutativity} $\mathsf{a_0a_1a_2}\subseteq \mathsf{b_0b_1b_2}\vdash \mathsf{a_1a_0a_2}\subseteq \mathsf{b_1b_0b_2}$.
 
 \item \label{prop:derivability_results_weakening} $\mathsf{a_0a_1}\subseteq \mathsf{b_0b_1}\vdash \mathsf{a_0a_0a_1}\subseteq \mathsf{b_0b_0b_1}$.

 \item \label{prop:derivability_results_contraction} $\mathsf{ab}\subseteq \mathsf{cd}\vdash \mathsf{a}\subseteq\mathsf{c}$.

 \item \label{prop:derivabilility_results_monotonicity} 
 $\mathsf{a}\subseteq\mathsf{b}, \alpha(\mathsf{b})\vdash\alpha(\mathsf{a})$, where $\alpha$, or $\alpha(\mathsf{c})$, is classical and propositional, the propositional symbols and constants occurring in $\alpha(\mathsf{c})$ are among those in $\mathsf{c}$, and $\alpha(\mathsf{a})$ and $\alpha(\mathsf{b})$ denote the result of replacing each element of $\mathsf{c}$ in $\alpha$ with the corresponding element in $\mathsf{a}$ and $\mathsf{b}$, respectively.
\end{enumerate}
\end{proposition}
\begin{proof}
\begin{enumerate}[label=(\roman*)] 

    \item By $\subseteq$Ext, it suffices to derive $\neg\mathsf{a}^\mathsf{x}\lor\mathsf{x}\subseteq\mathsf{c}$ for all $\mathsf{x}\in\{\top,\bot\}^{|\mathsf{a}|}$. For a given $\mathsf{x}$, we have $\mathsf{a}\subseteq\mathsf{b}\vdash\neg\mathsf{a}^\mathsf{x}\lor\mathsf{x}\subseteq\mathsf{b}$ and $\mathsf{b}\subseteq\mathsf{c}\vdash\neg\mathsf{b}^\mathsf{x}\lor\mathsf{x}\subseteq\mathsf{c}$ by $\subseteq$Rdt. By $\vee_{\subseteq}$E it now suffices to show:
     \begin{enumerate}[label=(\alph*)]
         \item $\neg\mathsf{a}^\mathsf{x}\vdash\neg\mathsf{a}^\mathsf{x}\lor\mathsf{x}\subseteq\mathsf{c}$.
         \item $\mathsf{x}\subseteq\mathsf{b},\neg\mathsf{b}^\mathsf{x}\vdash\neg\mathsf{a}^\mathsf{x}\lor\mathsf{x}\subseteq\mathsf{c}$.
         \item $\mathsf{x}\subseteq\mathsf{c}\vdash\neg\mathsf{a}^\mathsf{x}\lor\mathsf{x}\subseteq\mathsf{c}$.
     \end{enumerate}
     We have that (a) and (c) follow by $\vee$I, and (b) by $\subseteq_\lnot$E.

\item By $\subseteq$Ext, it suffices to derive $\neg\mathsf{a_1a_0a_2}^\mathsf{x_1x_0x_2}\lor\mathsf{x_1x_0x_2}\subseteq\mathsf{b_1b_0b_2}$ for all $\mathsf{x_1x_0x_2}\in\{\top,\bot\}^{|\mathsf{a_1a_0a_2}|}$. For a given $\mathsf{x_1x_0x_2}$, we have $\mathsf{a_0a_1a_2}\subseteq\mathsf{b_0b_1b_2}\vdash\neg\mathsf{a_0a_1a_2}^\mathsf{x_0x_1x_2}\lor\mathsf{x_0x_1x_2}\subseteq\mathsf{b_0b_1b_2}$ by $\subseteq$Rdt. By Proposition \ref{prop:classical_completeness} we have $\lnot\mathsf{a_0a_1a_2}^{\mathsf{x_0x_1x_2}}\dashv\vdash \lnot \mathsf{a_1a_0a_2}^{\mathsf{x_1x_0x_2}}$, and therefore also $\neg\mathsf{a_0a_1a_2}^\mathsf{x_0x_1x_2}\lor\mathsf{x_0x_1x_2}\subseteq\mathsf{b_0b_1b_2}\vdash\neg\mathsf{a_1a_0a_2}^\mathsf{x_1x_0x_2}\lor\mathsf{x_0x_1x_2}\subseteq\mathsf{b_0b_1b_2}$ $(*)$. We also have $\mathsf{b_1b_0b_2}\subseteq \mathsf{b_1b_0b_2}$ by $\subseteq$Id, and then $\lnot \mathsf{b_1b_0b_2}^{\mathsf{x_1x_0x_2}}\lor \mathsf{x_1x_0x_2}\subseteq\mathsf{b_1b_0b_2}$ by $\subseteq$Rdt. By Proposition \ref{prop:classical_completeness}, we have $\lnot \mathsf{b_0b_1b_2}^{\mathsf{x_0x_1x_2}} \dashv\vdash \lnot \mathsf{b_1b_0b_2}^{\mathsf{x_1x_0x_2}} $ so that $\lnot \mathsf{b_1b_0b_2}^{\mathsf{x_1x_0x_2}}\lor \mathsf{x_1x_0x_2}\subseteq\mathsf{b_1b_0b_2}\vdash\lnot \mathsf{b_0b_1b_2}^{\mathsf{x_0x_1x_2}}\lor \mathsf{x_1x_0x_2}\subseteq\mathsf{b_1b_0b_2}$ $(**)$. By $\lor_\subseteq$E applied to $(*)$ and $(**)$ it now suffices to show:
\begin{enumerate}
    \item $\neg\mathsf{a_1a_0a_2}^\mathsf{x_1x_0x_2}\vdash \neg\mathsf{a_1a_0a_2}^\mathsf{x_1x_0x_2}\lor\mathsf{x_1x_0x_2}\subseteq\mathsf{b_1b_0b_2}$.
    \item $\mathsf{x_0x_1x_2}\subseteq\mathsf{b_0b_1b_2}, \lnot \mathsf{b_0b_1b_2}^{\mathsf{x_0x_1x_2}}\vdash \neg\mathsf{a_1a_0a_2}^\mathsf{x_1x_0x_2}\lor\mathsf{x_1x_0x_2}\subseteq\mathsf{b_1b_0b_2}$.
    \item $ \mathsf{x_1x_0x_2}\subseteq\mathsf{b_1b_0b_2} \vdash \neg\mathsf{a_1a_0a_2}^\mathsf{x_1x_0x_2}\lor\mathsf{x_1x_0x_2}\subseteq\mathsf{b_1b_0b_2}$.
\end{enumerate}
We have that (a) and (c) follow by $\vee$I, and (b) by $\subseteq_{\lnot}$E.

     \item Similar to \ref{prop:derivability_results_commutativity}.

        \item By $\subseteq$Ext, it suffices to derive $\neg\mathsf{a}^\mathsf{x}\lor\mathsf{x}\subseteq\mathsf{c}$ for all $\mathsf{x}\in\{\top,\bot\}^{|\mathsf{a}|}$. For a given $\mathsf{x}$, we have $\bigwedge_{\mathsf{y}\in \{\top,\bot\}^{|\mathsf{b}|}}(\lnot \mathsf{ab}^{\mathsf{xy}}\lor \mathsf{xy}\subseteq \mathsf{cd})$ by $\subseteq$Rdt. By repeatedly applying $\vee_{\subseteq}$E, it suffices to show:
        \begin{enumerate}
        \item $\mathsf{xy}\subseteq\mathsf{cd}\vdash \neg\mathsf{a}^\mathsf{x}\lor\mathsf{x}\subseteq\mathsf{c}$ for any given $\mathsf{y}\in \{\top,\bot\}^{|\mathsf{b}|}$.
        \item $\bigwedge_{\mathsf{y}\in \{\top,\bot\}^{|\mathsf{b}|}}\lnot \mathsf{ab}^{\mathsf{xy}}\vdash  \neg\mathsf{a}^\mathsf{x}\lor\mathsf{x}\subseteq\mathsf{c}$.
        \end{enumerate}
        For (b), by Proposition \ref{prop:classical_completeness} we have $\vdash \bigvee_{\mathsf{y}\in \{\top,\bot\}^{|\mathsf{b}|}}\mathsf{b}^{\mathsf{y}}$, and by Proposition \ref{prop:classical_completeness} and $\vee$I, $\bigwedge_{\mathsf{y}\in \{\top,\bot\}^{|\mathsf{b}|}}\lnot \mathsf{ab}^{\mathsf{xy}}, \bigvee_{\mathsf{y}\in \{\top,\bot\}^{|\mathsf{b}|}}\mathsf{b}^{\mathsf{y}}\vdash  \neg\mathsf{a}^\mathsf{x}\lor\mathsf{x}\subseteq\mathsf{c}$ (for a given $\mathsf{y}$, $\lnot\mathsf{ab}^{\mathsf{xy}}=\lnot (\mathsf{a}^\mathsf{x}\land \mathsf{b}^\mathsf{y})\dashv \vdash \lnot \mathsf{a}^\mathsf{x}\lor\lnot \mathsf{b}^\mathsf{y}$, and $\lnot \mathsf{a}^\mathsf{x}\lor\lnot \mathsf{b}^\mathsf{y},\mathsf{b}^\mathsf{y}\vdash \lnot \mathsf{a}^\mathsf{x}$).

        For (a), we have $\mathsf{c}\subseteq \mathsf{c}$ by $\subseteq$Id, and then $\lnot \mathsf{c}^{\mathsf{x}}\lor \mathsf{x}\subseteq\mathsf{c}$ by $\subseteq$Rdt. By Proposition \ref{prop:classical_completeness}, we have $\lnot \mathsf{c}^{\mathsf{x}}\lor \mathsf{x}\subseteq\mathsf{c}\vdash \bigwedge_{\mathsf{z}\in \{\top,\bot\}^{|\mathsf{b}|}}\lnot \mathsf{cd}^{\mathsf{xz}}\lor \mathsf{x}\subseteq\mathsf{c}$ (for a given $\mathsf{z}$, we have $\lnot \mathsf{c}^\mathsf{x} \vdash \lnot \mathsf{c}^\mathsf{x}\lor \lnot \mathsf{d}^\mathsf{z}\dashv\vdash \lnot (\mathsf{c}^\mathsf{x}\land\mathsf{d}^\mathsf{z})=\lnot \mathsf{cd}^{\mathsf{xz}}$). By $\vee_\subseteq$E it therefore suffices to show:
        \begin{enumerate}\addtocounter{enumii}{2}
        \item $\mathsf{x}\subseteq\mathsf{c} \vdash \neg\mathsf{a}^\mathsf{x}\lor \mathsf{x}\subseteq\mathsf{c}$. 
        \item $\mathsf{xy}\subseteq \mathsf{cd}, \bigwedge_{\mathsf{z}\in \{\top,\bot\}^{|\mathsf{b}|}}\lnot \mathsf{cd}^{\mathsf{xz}} \vdash \neg\mathsf{a}^\mathsf{x}\lor\mathsf{x}\subseteq\mathsf{c}$ for any given $\mathsf{y}\in \{\top,\bot\}^{|\mathsf{b}|}$.
        \end{enumerate}
        We have that (c) follows by $\vee$I, and (d) by $\subseteq_\lnot $E.
        
    \item Let $n:=|\mathsf{a}|=|\mathsf{b}|=|\mathsf{c}|$. Define a function $f:\{\bot,\top\}^n\to \{\bot,\top\}$ by $f(\mathsf{y})=\top :\iff w\models \alpha(\mathsf{\mathsf{y}})$ for some $w$ (note that since the propositional symbols and constants appearing in $\alpha(\mathsf{c})$ are among those in $\mathsf{c}$, the definition of $f$ is independent of the choice of $w$). From $\mathsf{a}\subseteq \mathsf{b}$ we have $\bigwedge_{\{\mathsf{x}\in \{\top,\bot\}^n\mid f(\mathsf{x})=\bot\}}(\lnot \mathsf{a}^\mathsf{x}\lor \mathsf{x}\subseteq\mathsf{b})$ by $\subseteq$Rdt. By repeatedly applying $\vee_{\subseteq}$E, it suffices to show:
\begin{enumerate}[label=(\alph*)]
    \item $\bigwedge_{\{\mathsf{x}\in \{\top,\bot\}^n\mid f(\mathsf{x})=\bot\}}\lnot \mathsf{a}^\mathsf{x}\vdash \alpha(\mathsf{a})$.
    \item $\alpha(\mathsf{b}),\mathsf{x}\subseteq\mathsf{b}\vdash \alpha(\mathsf{a})$ for any given $\mathsf{x}\in \{\top,\bot\}^n$ with $f(\mathsf{x})=\bot$.
\end{enumerate}
To show (a) and (b), we show (c):
\begin{enumerate}[label=(\alph*)]\addtocounter{enumii}{2}
    \item $\bigwedge_{\{\mathsf{x}\in \{\top,\bot\}^n\mid f(\mathsf{x})=\bot\}}\lnot \mathsf{c}^\mathsf{x}\equiv \alpha(\mathsf{c})$.
\end{enumerate}
For the left-to-right direction, let $w\models\bigwedge_{\{\mathsf{x}\in \{\top,\bot\}^n \mid f(\mathsf{x})=\bot\}}\lnot \mathsf{c}^\mathsf{x}$ and assume for contradiction that $w\not\models \alpha (\mathsf{c})$. Then for $\mathsf{y}\in \{\bot,\top\}^n$ given by $y_i=\top \iff w\models c_i$ we have $w\not\models\alpha(\mathsf{y})$ so $f(\mathsf{y})=\bot$, whence $w\models \lnot \mathsf{c}^\mathsf{y}$. But clearly by the definition of $\mathsf{y}$ we have $w\models \mathsf{c}^\mathsf{y}$, a contradiction.

For the right-to-left direction, let $w\models \alpha(\mathsf{c})$ and let $\mathsf{x}\in \{\top,\bot\}^n$ be such that $f(\mathsf{x})=\bot$. Then for some $x_i$ we must have $w\models c_i \iff x_i=\bot$ (lest we have $w\models \alpha(\mathsf{x})$ whence $f(\mathsf{x})=\top$, contradicting $f(\mathsf{x})=\bot$), and therefore $w\models \lnot \mathsf{c}^\mathsf{x}$.

Given (c), (a) follows immediately by Proposition \ref{prop:classical_completeness}.

For (b), by (c) and Proposition \ref{prop:classical_completeness}, $\alpha(\mathsf{b})$ gives us $\lnot \mathsf{b}^\mathsf{x}$. By $\subseteq_{\lnot}$E applied to $\lnot \mathsf{b}^\mathsf{x}$ and $\mathsf{x}\subseteq\mathsf{b}$, we have $\alpha(\mathsf{a})$. \qedhere 

\end{enumerate}
\end{proof}

We move on to the completeness proof. Our strategy 
involves showing that each $\ML(\subseteq)$-formula is provably equivalent to a formula in the normal form presented in Section \ref{section:expressive_power}. Once all formulas are in normal form, completeness follows from the semantic and proof-theoretic properties of formulas in this form. That is, we show: 

\begin{lemma}[Provable equivalence of the normal form]\label{nfproveeq}
For any formula $\phi(\mathsf{X})$ in $\ML(\subseteq)$ and $k\geq md(\phi)$, there is a (finite, nonempty) property $\mathcal{C}$ (over $\mathsf{X}$) such that
$$\phi\dashv \vdash\bigvee_{(M,T)\in\mathcal{C}}\theta^{\mathsf{X},k}_{M,T},\quad\quad\text{i.e.,}\quad\quad\phi \dashv \vdash\bigvee_{(M,T)\in\mathcal{C}}(\bigvee_{w\in T}\chi^{\mathsf{X},k}_{M,w}\land\bigwedge_{w\in T}(\top\subseteq\chi^{\mathsf{X},k}_{M,w})).$$
\end{lemma}

The proof of the above lemma is more involved. We postpone it for now, and first focus on how it allows us to to prove completeness. To that end, note the following immediate consequence of classical completeness: the $k$th Hintikka formulas of two $k$-bisimilar pointed models are provably equivalent.

\begin{lemma}\label{lemma:hintikka_provable_equivalence} If $M,w\leftrightarroweq_k M^\prime,u$, then $\chi_{M,w}^k\dashv\vdash\chi_{M^\prime,u}^k$.
\end{lemma}
\begin{proof}
By Theorem \ref{thm:world_bisimulation},  $w\leftrightarroweq_k u$ implies $w\equiv_k u$ and then $\chi_{w}^k\dashv\vdash\chi_{u}^k$ by Proposition \ref{prop:classical_completeness}.
\end{proof}

 Recall from footnote \ref{footnote:finite_representatives} in Section \ref{section:expressive_power} that in defining the $(k+1)$th-Hintikka formula $\chi^{\mathsf{X},k+1}_{M,w}$, we choose an arbitrary finite set $T$ of representatives from the possibly infinite set $R[w]$. The above lemma shows that two $k$-bisimilar worlds give rise to provably equivalent Hintikka formulas, so the specific choice of the representatives also makes no difference in our proof system. Therefore, we may  continue our practice of using arbitrary finite representatives in the definition of $\chi^{\mathsf{X},k+1}_{M,w}$ in the context of the proof system. Similar results hold for other formulas employing finite representatives, though we henceforth omit explicit discussion of these representatives. Our next aim is to prove the analogous result for Hintikka formulas for teams---that is:

\begin{lemma}\label{lemma:team_hintikka_provable_equivalence}
If $M,T\leftrightarroweq_k M^\prime,S$, then $\theta_{M,T}^k\dashv\vdash\theta_{M^\prime,S}^k$. 
\end{lemma}

In order to establish this, we first prove some results concerning primitive inclusion atoms. We also 
show another unrelated result concerning primitive inclusion atoms (\ref{lem:top_atom_derivability_results_bot_top_exc} in the lemma below) which 
is required 
in the sequel. 

\begin{lemma}\label{lem:top_atom_derivability_results}\
\begin{enumerate}[label=(\roman*)]
 \item \label{lem:top_atom_derivability_results_i}$\alpha_1,\dots,\alpha_n\vdash\top^n\subseteq\alpha_1\dots\alpha_n$, where $\top^n$ denotes the sequence $\top,\ldots,\top$ of length $n$.
 \item \label{lem:top_atom_derivability_results_ii}If $\alpha_i\vdash\beta_i$, then $\top^n\subseteq\alpha_1\dots\alpha_i\dots\alpha_n\vdash\top^n\subseteq\alpha_1\dots\beta_i\dots\alpha_n$. 
 \item \label{lem:top_atom_derivability_results_bot_top_exc} $\top^n\subseteq \alpha_{1}^{x_1}\ldots \alpha_n^{x_n}\vdash x_1\ldots x_n\subseteq \alpha_1\ldots \alpha_n$.
\end{enumerate}
\end{lemma}
\begin{proof}
\begin{enumerate}[label=(\roman*)] 
    \item $\top \subseteq \top$ by $\subseteq$Id; using $\alpha_n$ and $\top\subseteq \top$ we get $\top\top\subseteq \alpha_n\top$ by $\subseteq$Exp; and then $\top \top\subseteq \alpha_n\top\vdash \top\subseteq \alpha_n$ by Proposition \ref{prop:derivability_results} \ref{prop:derivability_results_contraction}. We repeat the $\subseteq$Exp-step with $\alpha_{n-1},\ldots, \alpha_1$.
    
 \item By Proposition \ref{prop:classical_completeness}, $\vdash\neg(\alpha_1\land\dots\land\alpha_n)\lor(\alpha_1\land\dots\land\alpha_n)$. From the assumption $\alpha_i\vdash\beta_i$ together with item \ref{lem:top_atom_derivability_results_i} we have that $\alpha_1\land\dots\land\alpha_n\vdash\alpha_1\land\dots\land\beta_i\land\dots\land\alpha_n\vdash\top^n\subseteq\alpha_1\dots\beta_i\dots\alpha_n$. Thus $\neg(\alpha_1\land\dots\land\alpha_n)\lor(\alpha_1\land\dots\land\alpha_n)\vdash\neg(\alpha_1\land\dots\land\alpha_n)\lor(\top^n\subseteq\alpha_1\dots\beta_i\dots\alpha_n)$. By $\lor_\subseteq$E, it suffices to show: 
      \begin{enumerate}[label=(\alph*)]
         \item $\top^n\subseteq\alpha_1\dots\alpha_n,\neg(\alpha_1\land\dots\land\alpha_n) \vdash \top^n\subseteq\alpha_1\dots\beta_i\dots\alpha_n$
         
         \item $\top^n\subseteq\alpha_1\dots\alpha_n ,\top^n\subseteq\alpha_1\dots\beta_i\dots\alpha_n\vdash \top^n\subseteq\alpha_1\dots\beta_i\dots\alpha_n$
     \end{enumerate}
We have that (a) follows by $\subseteq_\neg$E, and (b) is immediate.

        \item By $\subseteq$Id, we have $\alpha_1\ldots\alpha_n\subseteq \alpha_1\ldots\alpha_n$ and then by $\subseteq$Rdt, $\lnot (\alpha_1\ldots\alpha_n^{x_1\ldots x_n})\lor x_1\ldots x_n \subseteq \alpha_1\ldots \alpha_n$. By $\vee_\subseteq$E it suffices to show $\lnot (\alpha_1\ldots\alpha_n^{x_1\ldots x_n}),\top^n\subseteq \alpha_1^{x_1}\ldots \alpha_n^{x_n}\vdash x_1\ldots x_n\subseteq \alpha_1\ldots \alpha_n$ and $x_1\ldots x_n \subseteq \alpha_1\ldots \alpha_n\vdash x_1\ldots x_n\subseteq \alpha_1\ldots \alpha_n$. The latter is immediate; the former follows by $\subseteq_\lnot$E since $\alpha_1\ldots\alpha_n^{x_1\ldots x_n}=(\alpha_1^{x_1}\ldots \alpha_n^{x_n})^{\top^n}$.\qedhere
\end{enumerate}
\end{proof}

\begin{proof}[Proof of Lemma \ref{lemma:team_hintikka_provable_equivalence}]
Suppose that $T\leftrightarroweq_k S$. If $T=S=\emptyset$, then $\theta_{T}^k=\theta_{S}^k$. Otherwise, both $T$ and $S$ are nonempty. By symmetry it suffices to show $\theta_{T}^k\vdash\theta_{S}^k$. We have that $\theta_{T}^k\vdash\bigvee_{w\in T}\chi_{w}^k$ by $\land$E. Let $w\in T$. By $T\leftrightarroweq_k S$, there is some $u\in S$ such that $w\leftrightarroweq_k u$, so that by Lemma \ref{lemma:hintikka_provable_equivalence}, $\chi^k_w\dashv\vdash \chi^k_u$. Therefore, by $\vee$I and $\vee$E, $\bigvee_{w\in T}\chi_{w}^k\vdash\bigvee_{u\in S}\chi_{u}^k$. Now let $u\in S$. By $T\leftrightarroweq_k S$, there is some $w_u\in T$ such that $w_u\leftrightarroweq_k u$, so that by Lemma \ref{lemma:hintikka_provable_equivalence}, $\chi^k_{w_u}\dashv\vdash \chi^k_u$. By $\land$E and Lemma \ref{lem:top_atom_derivability_results} \ref{lem:top_atom_derivability_results_ii}, $\theta_{T}^k\vdash\top\subseteq\chi_{w_u}^k\vdash\top\subseteq\chi_{u}^k$. Therefore by $\land$I, $\theta_{T}^k\vdash\bigwedge_{u\in S}(\top\subseteq\chi_{u}^k)$, so that $\theta_{T}^k\vdash\bigvee_{u\in S}\chi_{u}^k\land\bigwedge_{u\in S}(\top\subseteq\chi_{u})=\theta_{S}^k$. 
\end{proof}

Next we prove some important semantic facts concerning the relationship between formulas in normal form and disjoint unions. The \emph{disjoint union} $\biguplus_{i\in I} M_i $ of the models $M_i$ ($i\in I$) is defined as usual, where in particular we recall that the domain of the disjoint union is defined to be $\biguplus_{i\in I}W_i=\bigcup_{i\in I}(W_i \times \{i\})$. We also extend the notion to properties: the \emph{disjoint union} of a property $\mathcal{C}=\{(M_i,T_i)\mid i\in I\}$ is $\biguplus \mathcal{C}=(\biguplus_{i\in I}M_i, \biguplus_{i\in I}T_i)$, where $\biguplus_{i\in I}T_i=\bigcup_{i\in I}(T_i\times \{i\})$. To simplify notation, we define the disjoint union $\biguplus \emptyset$ of the empty property to be $(M^*,\emptyset)$ for some fixed model $M^*$.
We refer to a world $(w,i)$ in a disjoint union as simply $w$, and  to a team  $T \times\{i\}$ as simply $T$. The following result is standard (see, for instance, \cite{blackburn2001}):
\begin{proposition} \label{prop:disjoint_union}
    For any $i\in I$, any $w\in W_i$, and any $k\in \mathbb{N}$: $M_i,w \leftrightarroweq_k \biguplus_{i\in I} M_i, w $. Therefore, for any $T\subseteq W_i$: $M_i,T\leftrightarroweq_k \biguplus_{i\in I} M_i, T $.
\end{proposition}
We have: 

\begin{lemma} \label{lemma:nf_disjoint_union} 
$M,S\models\bigvee_{(M',T)\in\mathcal{C}}\theta^k_{M^\prime,T}\text{ iff }M,S\leftrightarroweq_k\biguplus\mathcal{C}' \text{ for some }\mathcal{C}'\subseteq \mathcal{C}.$
\end{lemma}
\begin{proof}
$\Longrightarrow$: If $M,S\models \bigvee_{(M^\prime,T)\in\mathcal{C}}\theta^k_{M^\prime,T}$, then $mS=\bigcup_{(M',T)\in \mathcal{C}}S_{M',T}$ where $S_{M',T}\models \theta^k_{M',T}$. By Lemma \ref{lemma:team_bisimulation} \ref{lemma:team_bisimulation_iii}, either $S_{M',T}\leftrightarroweq_k T$ or $S_{M',T}=\emptyset$. Let $\mathcal{C}':=\{(M',T)\mid S_{M',T}\leftrightarroweq_k T\}$. Then $S=\bigcup_{(M',T)\in \mathcal{C}}S_{M',T}=\bigcup_{(M',T)\in \mathcal{C}'}S_{M',T}$. By Proposition \ref{prop:disjoint_union}, we have:
$$M,S=M,\bigcup_{(M',T)\in\mathcal{C}'}S_{M',T}\leftrightarroweq_k\biguplus\{(M',T)\mid T\in \mathcal{C}'\}=\biguplus \mathcal{C}'.$$
Note that if $S_{M',T}=\emptyset$ for all $(M',T)\in \mathcal{C}$, then $S=\emptyset$ and $\biguplus \mathcal{C}'=\biguplus \emptyset=(M^*,\emptyset)$ by our convention.

$\Longleftarrow$: Let $\mathcal{C}'=\{(M_i,T_i)\mid i \in I\}$. If $M,S\leftrightarroweq_k \biguplus\mathcal{C}'$, then by Lemma \ref{lemma:bisimulation_properties} \ref{lemma:bisimulation_properties_iii}, there are subteams $S_i\subseteq S$ such that $\bigcup_{i\in I}S_i=S$ and $M,S_i\leftrightarroweq_k \biguplus_{j \in I}M_j,T_i$. By Proposition \ref{prop:disjoint_union}, $\biguplus_{j \in I}M_j,T_i\leftrightarroweq_k M_i,T_i$ so that also $M,S_i\leftrightarroweq_k M_i,T_i$. By Lemma \ref{lemma:team_bisimulation} \ref{lemma:team_bisimulation_iii}, $M,S_i\models \theta^k_{M_i,T_i}$ so that $M,S\models \bigvee_{(M_i,T_i)\in \mathcal{C'}}\theta^k_{M_i,T_i}$, and by the empty team property, $M,S\models \bigvee_{(M^\prime,T)\in \mathcal{C}}\theta^k_{M^\prime,T}$.
\end{proof}

\begin{cor}\label{coro:disjoint_union_entailment}
    $\bigvee_{(M,T)\in\mathcal{C}}\theta^k_{M,T}\models\bigvee_{(M^\prime,S)\in\mathcal{D}}\theta^k_{M^\prime,S}$ iff for all $(M^{\prime\prime},U)\in\mathcal{C}$, $M^{\prime\prime},U\leftrightarroweq_k \biguplus\mathcal{D}_{(M'',U)}$ for some $\mathcal{D}_{(M'',U)}\subseteq\mathcal{D}$.
\end{cor}
\begin{proof}
    $\Longrightarrow$: Let $(M'',U)\in \mathcal{C}$. By Lemma \ref{lemma:team_bisimulation} \ref{lemma:team_bisimulation_iii}, $U\models \theta^k_U$. By the empty team property, $U \models \bigvee_{{(M,T)}\in\mathcal{C}}\theta^k_{T}$, and then by assumption, $U\models \bigvee_{{(M',S)}\in\mathcal{D}}\theta^k_{S}$. The result now follows by Lemma \ref{lemma:nf_disjoint_union}.

    $\Longleftarrow$: Let $M''',Y\models \bigvee_{{(M,T)}\in\mathcal{C}}\theta^k_{T}$. By Lemma \ref{lemma:nf_disjoint_union}, $M''',Y\leftrightarroweq_k \biguplus \mathcal{C}'$ for some $\mathcal{C}'\subseteq\mathcal{C}$. By assumption, for each $(M'',U)\in \mathcal{C'}$ there is some $\mathcal{D}_{(M'',U)}\subseteq\mathcal{D}$ such that $M'',U\leftrightarroweq_k \biguplus\mathcal{D}_{(M'',U)}$. Then $M''',Y\leftrightarroweq_k \biguplus \{\biguplus \mathcal{D}_{(M'',U)}\mid (M'',U)\in \mathcal{C}'\}\leftrightarroweq_k \biguplus \bigcup_{(M'',U)\in \mathcal{C}'}\mathcal{D}_{(M'',U)} $, so that by Lemma \ref{lemma:nf_disjoint_union}, $Y\models \bigvee_{{(M',S)}\in\mathcal{D}}\theta^k_{S}$.
\end{proof}

The final lemma required for the completeness proof is a proof-theoretic analogue of (one direction of) Lemma \ref{lemma:nf_disjoint_union}.

\reqnomode
\begin{lemma}\label{lemma:disjoint_union_provability}
 $\theta^k_{\biguplus\mathcal{D}}\vdash\bigvee_{(M,S)\in\mathcal{D}}\theta_{M,S}^k$. 
\end{lemma}
\begin{proof}
If $\mathcal{D}$ is empty, then $\theta^k_{\biguplus\mathcal{D}}\dashv\vdash\bot \vdash\bigvee_{(M,S)\in\mathcal{D}}\theta_{S}^k$. Otherwise let $\mathcal{D}=\{S_1,\ldots ,S_n\}$.
We have: 
\begin{equation*}
\theta^k_{\biguplus\mathcal{D}} 
=\bigvee_{w\in \biguplus\mathcal{D}}\chi^k_{w}\land\bigwedge_{w\in \biguplus\mathcal{D}}(\top\subseteq \chi^k_{w}) 
=\bigvee_{(M,S)\in\mathcal{D}}\bigvee_{w\in S}\chi^k_{w}\land\bigwedge_{(M,S)\in\mathcal{D}}\bigwedge_{w \in S}(\top\subseteq \chi^k_{w}),
\end{equation*}

from which we derive:    
    \begin{align*}
    &&&\vdash&& (\bigvee_{w\in S_1}\chi^k_{w}\vee\bigvee_{2\leq i\leq n}\bigvee_{w\in S_i}\chi^k_{w})\land \bigwedge_{w \in S_1}(\top\subseteq \chi^k_{w})\land \bigwedge_{2\leq i\leq n}\bigwedge_{w \in S_i}(\top\subseteq \chi^k_{w})\\
    &&&\vdash&& (((\bigvee_{w\in S_1}\chi^k_{w}\vee\bigvee_{w\in S_1}\chi^k_{w})\land \bigwedge_{w \in S_1}(\top\subseteq \chi^k_{w}))\vee\bigvee_{2\leq i\leq n}\bigvee_{w\in S_i}\chi^k_{w})\land \bigwedge_{2\leq i\leq n}\bigwedge_{w \in S_i}(\top\subseteq \chi^k_{w})\tag{$\subseteq $Distr}\\
    &&&\vdash&& (\theta^k_{S_1}\vee\bigvee_{2\leq i\leq n}\bigvee_{w\in S_i}\chi^k_{w})\land \bigwedge_{2\leq i\leq n}\bigwedge_{w \in S_i}(\top\subseteq \chi^k_{w})\\
    &&&\vdots &&\\
   &&&\vdash &&\theta^k_{S_1}\vee\ldots \vee \theta^k_{S_n}.
\end{align*}
\leqnomode
\end{proof}

\begin{theorem}[Completeness]\label{Completeness}
If $\Gamma\models\psi$, then $\Gamma\vdash\psi$.
\end{theorem}

\begin{proof}
 Suppose that $\Gamma\models\psi$. Since $\ML(\subseteq)$ is compact, there is a finite subset $\Gamma_0\subseteq\Gamma$ such that $\Gamma_0\models\psi$. It suffices to show that $\phi\vdash\psi$ where $\phi=\bigwedge_{\gamma\in\Gamma_0}\gamma$.
 
Let $k\geq$max$\{ md(\phi),$ $md(\psi)\}$. By Lemma \ref{nfproveeq},
\begin{equation*}
    \phi\dashv\vdash\bigvee_{(M,T)\in\mathcal{C}}\theta^k_{M,T} \text{ and } \psi\dashv\vdash\bigvee_{(M^\prime,S)\in\mathcal{D}}\theta^k_{M^\prime,S}.
\end{equation*}
 for some finite nonempty properties $\mathcal{C}$ and $\mathcal{D}$. By soundness and $\phi\models \psi$, 
$$\bigvee_{(M,T)\in\mathcal{C}}\theta^k_{M,T}\models\bigvee_{(M^\prime,S)\in\mathcal{D}}\theta^k_{M^\prime,S}.$$ 
Let $(M,T)\in \mathcal{C}$. By Corollary \ref{coro:disjoint_union_entailment}, $M,T\leftrightarroweq_k \biguplus{\mathcal{D}_T}$ for some $\mathcal{D}_T\subseteq\mathcal{D}$. By Lemma \ref{lemma:team_hintikka_provable_equivalence}, Lemma \ref{lemma:disjoint_union_provability}, and $\lor$I, $$\theta^k_{M,T}\vdash\theta^k_{\biguplus\mathcal{D}_T}\vdash\bigvee_{(M^\prime,S)\in\mathcal{D}_T}\theta^k_{M^\prime,S}\vdash\bigvee_{(M^\prime,S)\in\mathcal{D}}\theta^k_{M^\prime,S}.$$
By $\lor$E we get that $\bigvee_{(M,T)\in\mathcal{C}}\theta^k_{M,T}\vdash\bigvee_{(M^\prime,S)\in\mathcal{D}}\theta^k_{M^\prime,S}$, and conclude that $\phi\vdash\psi$.
\end{proof}

We dedicate the rest of this subsection to the proof of Lemma \ref{nfproveeq}: provable equivalence of the normal form. We first prove some technical lemmas. The rules $\vee_\subseteq$E and $\Box\vee_{\subseteq}$E can be generalized as follows:

\begin{lemma} 
\label{lemma:inclusion_disjunction_elimination_generalization}
\begin{enumerate}[label=(\roman*)]
   \item \label{lemma:inclusion_disjunction_elimination_generalization_non_modal_i} If
    \begin{enumerate}[label=(\alph*)]
        \item $\Gamma, \mathsf{x}_1\subseteq \mathsf{a}_1,\dots, \mathsf{x}_k\subseteq \mathsf{a}_k\vdash\chi$ and
        \item $\Gamma,\psi\vdash\chi$,
    \end{enumerate}
then $\Gamma,(\mathsf{x}_1\subseteq \mathsf{a}_1\land\dots\land \mathsf{x}_k\subseteq \mathsf{a}_k)\lor\psi\vdash\chi$.
\item \label{lemma:inclusion_disjunction_elimination_generalization_non_modal} Let $I$ be a finite index set and for each $i\in I$, let $\iota_i$ be a conjunction of finitely many primitive inclusion atoms. If for every nonempty $J \subseteq I$,
    \begin{equation*}
         \quad \Gamma,\bigvee_{j\in J}\phi_j,\bigwedge_{j\in J}\iota_j\vdash\chi, 
    \end{equation*}
then $\Gamma,\bigvee_{i\in I}(\phi_i\land\iota_i)\vdash\chi.$
   \item \label{lemma:inclusion_disjunction_elimination_generalization_modal_i} If
    \begin{enumerate}[label=(\alph*)]
        \item $\Gamma, \top\subseteq \Diamond\mathsf{a}_1,\dots, \top\subseteq \Diamond\mathsf{a}_k\vdash\chi$ and
        \item $\Gamma,\Box\psi\vdash\chi$,
    \end{enumerate}
then $\Gamma,\Box((\mathsf{x}_1\subseteq \mathsf{a}_1\land\dots\land \mathsf{x}_k\subseteq \mathsf{a}_k)\lor\psi)\vdash\chi$.
\item \label{lemma:inclusion_disjunction_elimination_generalization_modal} Let $I$ be a finite index set and for each $i\in I$, let $\iota_i$ be a conjunction of finitely many primitive inclusion atoms. For $\iota_i=\bigwedge_{k\in K_i} (\mathsf{x}_k\subseteq\mathsf{a}_k)$, define $\iota_{\Diamond i}=\bigwedge_{k\in K_i} (\top\subseteq\Diamond\mathsf{a}_k^{\mathsf{x}_k})$. If for every nonempty $J \subseteq I$
    \begin{equation*}  
        \Gamma,\Box\bigvee_{j\in J}\phi_j,\bigwedge_{j\in J}\iota_{\Diamond j}\vdash\chi,
      \end{equation*}
then $\Gamma,\Box\bigvee_{i\in I}(\phi_i\land\iota_i)\vdash\chi.$
\end{enumerate}
\end{lemma}
\begin{proof}
We prove \ref{lemma:inclusion_disjunction_elimination_generalization_non_modal_i} and \ref{lemma:inclusion_disjunction_elimination_generalization_non_modal}; the proofs of \ref{lemma:inclusion_disjunction_elimination_generalization_modal_i} and \ref{lemma:inclusion_disjunction_elimination_generalization_modal} are similar.
\begin{enumerate}[label=(\roman*)]
    \item We have that $(\mathsf{x}_1\subseteq\mathsf{a}_1\land\dots\land \mathsf{x}_k\subseteq \mathsf{a}_k)\lor\psi\vdash ((\mathsf{x}_1\subseteq \mathsf{a}_1\land\dots\land \mathsf{x}_{k-1}\subseteq \mathsf{a}_{k-1}) \vee\psi)\land (\mathsf{x}_k\subseteq \mathsf{a}_k\vee\psi)$ by $\vee$E, $\vee$I, $\land $E, and $\land$I. To show $\Gamma,(\mathsf{x}_1\subseteq \mathsf{a}_1\land\dots\land \mathsf{x}_{k-1}\subseteq \mathsf{a}_{k-1}) \vee\psi, \mathsf{x}_k\subseteq \mathsf{a}_k\vee \psi \vdash\chi$, by $\vee_\subseteq $E it suffices to show:
    \begin{enumerate}[label=(\alph*)]
        \item $\Gamma,(\mathsf{x}_1\subseteq \mathsf{a}_1\land\dots\land \mathsf{x}_{k-1}\subseteq \mathsf{a}_{k-1}) \vee\psi,\psi \vdash \chi$ and
        \item $\Gamma,(\mathsf{x}_1\subseteq \mathsf{a}_1\land\dots\land \mathsf{x}_{k-1}\subseteq \mathsf{a}_{k-1})\vee\psi,\mathsf{x}_k\subseteq \mathsf{a}_k\vdash\chi$.
    \end{enumerate}
    Showing $\Gamma,\psi \vdash \chi$ suffices for (a). As for (b), similarly to the above, we have $((\mathsf{x}_1\subseteq \mathsf{a}_1\land\dots\land \mathsf{x}_{k-1}\subseteq \mathsf{a}_{k-1})\vee\psi )\land \mathsf{x}_k\subseteq \mathsf{a}_k\vdash((\mathsf{x}_1\subseteq \mathsf{a}_1\land\dots\land \mathsf{x}_{k-2}\subseteq \mathsf{a}_{k-2})\vee\psi)\land \mathsf{x}_{k}\subseteq \mathsf{a}_k \land(\mathsf{x}_{k-1}\subseteq \mathsf{a}_{k-1}\vee\psi)$. To show (b), it therefore suffices, by $\vee_\subseteq$E, to show:
        \begin{enumerate}[label=(\alph*)] \setcounter{enumii}{2}
        \item $\Gamma,(\mathsf{x}_1\subseteq \mathsf{a}_1\land\dots\land \mathsf{x}_{k-2}\subseteq \mathsf{a}_{k-2})\vee\psi,\mathsf{x}_{k}\subseteq \mathsf{a}_k,\psi \vdash\chi$ and
        \item $\Gamma,(\mathsf{x}_1\subseteq \mathsf{a}_1\land\dots\land \mathsf{x}_{k-2}\subseteq \mathsf{a}_{k-2})\vee\psi,\mathsf{x}_{k}\subseteq \mathsf{a}_k,\mathsf{x}_{k-1}\subseteq \mathsf{a}_{k-1}\vdash\chi$.
    \end{enumerate}
    Showing $\Gamma,\psi \vdash \chi$ would again suffice for (c). Continuing in the same manner, one eventually finds that it suffices to show $\Gamma,\psi \vdash \chi$ and $\Gamma, \mathsf{x}_1\subseteq \mathsf{a}_1,\dots, \mathsf{x}_k\subseteq \mathsf{a}_k\vdash\chi$.

  \item Let $I=\{1,\ldots, n\}$. We have that $\bigvee_{1\leq i\leq n-1}(\phi_i\land \iota_i)\vdash (\bigvee_{1\leq i\leq n-1}(\phi_i\land \iota_i)\vee\phi_n) \land (\bigvee_{1\leq i\leq n-1}(\phi_i\land \iota_i)\vee\iota_n)$ by $\vee$E, $\vee$I, $\land $E, and $\land$I. To show $\Gamma,\bigvee_{1\leq i\leq n-1}(\phi_i\land \iota_i)\vee\phi_n, \bigvee_{1\leq i\leq n-1}(\phi_i\land \iota_i)\vee\iota_n\vdash\chi$, by \ref{lemma:inclusion_disjunction_elimination_generalization_non_modal_i} it suffices to show:
  \begin{enumerate}[label=(\alph*)]
      \item $\Gamma,\bigvee_{1\leq i\leq n-1}(\phi_i\land \iota_i)\vee\phi_n, \bigvee_{1\leq i\leq n-1}(\phi_i\land \iota_i)\vdash\chi$ and
      \item $\Gamma,\bigvee_{1\leq i\leq n-1}(\phi_i\land \iota_i)\vee\phi_n, \iota_n\vdash\chi$.
  \end{enumerate}
  By $\vee$I, showing $\Gamma,\bigvee_{1\leq i\leq n-1}(\phi_i\land \iota_i)\vdash\chi$ suffices for (a). Furthermore, by a similar manipulation as used above, (c) suffices for (a), and (d) suffices for (b).
    \begin{enumerate}[label=(\alph*)] \setcounter{enumii}{2}
        \item $\Gamma,\bigvee_{1\leq i\leq n-2}(\phi_i\land \iota_i)\vee\phi_{n-1},\bigvee_{1\leq i\leq n-2}(\phi_i\land \iota_i)\vee\iota_{n-1}\vdash\chi$.
        \item $\Gamma,\bigvee_{1\leq i\leq n-2}(\phi_i\land \iota_i)\vee \phi_{n-1}\vee\phi_n,\bigvee_{1\leq i\leq n-2}(\phi_i\land \iota_i)\vee\phi_n\vee \iota_{n-1}, \iota_n\vdash\chi$.
    \end{enumerate}
    By \ref{lemma:inclusion_disjunction_elimination_generalization_non_modal_i}, to show (c) it suffices to show (e) and (f), and to show (d) it suffices to show (g) and (h):
        \begin{enumerate}[label=(\alph*)] \setcounter{enumii}{4}
        \item $\Gamma,\bigvee_{1\leq i\leq n-2}(\phi_i\land \iota_i)\vee\phi_{n-1},\bigvee_{1\leq i\leq n-2}(\phi_i\land \iota_i)\vdash\chi$.
        \item $\Gamma,\bigvee_{1\leq i\leq n-2}(\phi_i\land \iota_i)\vee\phi_{n-1},\iota_{n-1}\vdash\chi$.
        \item $\Gamma,\bigvee_{1\leq i\leq n-2}(\phi_i\land \iota_i)\vee \phi_{n-1}\vee\phi_n,\bigvee_{1\leq i\leq n-2}(\phi_i\land \iota_i)\vee\phi_n, \iota_n\vdash\chi$.
        \item $\Gamma,\bigvee_{1\leq i\leq n-2}(\phi_i\land \iota_i)\vee \phi_{n-1}\vee\phi_n, \iota_{n-1}, \iota_n\vdash\chi$.
    \end{enumerate}
    Showing $\Gamma,\bigvee_{1\leq i\leq n-2}(\phi_i\land \iota_i)\vdash\chi$ would suffice to show (e), and showing $\Gamma,\bigvee_{1\leq i\leq n-2}(\phi_i\land \iota_i)\vee\phi_n, \iota_n\vdash\chi$ would suffice to show (g). Continuing in the same manner, one eventually finds that it suffices to show $\Gamma,\bigvee_{j\in J}\phi_j,\bigwedge_{j\in J}\iota_j\vdash\chi$ for every nonempty $J\subseteq I$. \qedhere
\end{enumerate}
\end{proof}

Next, we show that the Hintikka formulas of non-bisimilar pointed models/models with teams are contradictory in our proof system.

\begin{lemma}\label{lemma:non_bisimilar_hintikka_formulas_contradictory}\
\begin{enumerate}[label=(\roman*)]
    \item \label{lemma:non_bisimilar_hintikka_formulas_contradictory_world} If $M,w\not\leftrightarroweq_k M^\prime,u$, then $\chi_{M,w}^k,\chi_{M^\prime,u}^k\vdash\bot$.
    \item \label{lemma:non_bisimilar_hintikka_formulas_contradictory_team} If $M,T\not\leftrightarroweq_k M^\prime,S$, then $\theta_{M,T}^k,\theta_{M^\prime,S}^k\vdash\bot$.
\end{enumerate}
\end{lemma}
\begin{proof}
\begin{enumerate}[label=(\roman*)]
    \item Assume for contradiction that $\chi_{w}^k,\chi_{u}^k\not\vdash\bot$. By Proposition \ref{prop:classical_completeness}, there is some model $M^{\prime\prime}$ and a nonempty team $T$ of $M''$ such that $M^{\prime\prime},T\models\chi_{w}^k$ and $M^{\prime\prime},T\models\chi_{u}^k$. By flatness, $M^{\prime\prime},v\models\chi_{w}^k$ and $M^{\prime\prime},v\models\chi_{u}^k$ for all $v\in T$, from which it follows by Theorem \ref{thm:world_bisimulation} that $M,w\leftrightarroweq_k M^{\prime\prime}, v \leftrightarroweq_k M^\prime,u$, a contradiction.  
    \item W.l.o.g., we assume that 
    there is some $w\in T$ such that $M,w\not\leftrightarroweq_k M^\prime,u$ for all $u\in S$. By item \ref{lemma:non_bisimilar_hintikka_formulas_contradictory_world}, $\chi_{w}^k,\chi_{u}^k\vdash\bot$ for all $u\in S$, whence by $\lor$E, $\bigvee_{u\in S}\chi_{u}^k,\chi_{w}^k\vdash\bot$. By $\neg$I, we derive $\bigvee_{u\in S}\chi_{u}^k\vdash\neg\chi_{w}^k$. Then by $\land$E, we have $\theta_{S}^k\vdash\bigvee_{u\in S}\chi_{u}^k\vdash\neg\chi_{w}^k$. We also derive $\theta_{T}^k\vdash\top\subseteq\chi_{w}^k$ by $\land$E. Finally, we use $\subseteq_\neg$E to derive $\top\subseteq\chi_{w}^k,\neg\chi_{w}^k\vdash\bot$.\qedhere
\end{enumerate}
\end{proof}

Finally, we note the following simple consequence of classical completeness:
\begin{lemma}\label{lemma:world_satisfaction_to_hintikka_provability}
If $M,w\models\alpha$, then $\chi_{M,w}^k\vdash\alpha$, where $k\geq  md(\alpha)$.
\end{lemma}
\begin{proof}
By Proposition \ref{prop:classical_completeness}, it suffices to show $\chi_w^k\models\alpha$. Let $T\models\chi_w^k$ so that $u\models\chi_w^k$ for all $u\in T$ by flatness. By Theorem \ref{thm:world_bisimulation}, it follows that $w\equiv_k u$ whence also $u\models\alpha$ for all $u\in T$. Using flatness again, we conclude $T\models\alpha$.
\end{proof}

We are now ready to prove the main lemma.


\begin{proof}[Proof of Lemma \ref{nfproveeq}]
By induction on $\phi$.
\begin{enumerate} 


\item[\boldmath$\cdot$] If $\phi=\alpha \in \ML$ and $k\geq md(\alpha)$, letting $\mathcal{D}:={\{(M,w)\mid w\models \alpha\}}$, by a standard normal form result for $\ML$ (see, e.g., \cite{goranko}), we have that $\alpha\equiv\bigvee_{(M,w)\in \mathcal{D}}\chi_{w}^k$. Then by Proposition \ref{prop:classical_completeness}, $\alpha\dashv \vdash\bigvee_{(M,w)\in \mathcal{D}}\chi_{w}^k$. If $\mathcal{D}=\emptyset$, we have $\alpha\dashv \vdash\bigvee_{(M,w)\in \mathcal{D}}\chi_{w}^k = \bot = \theta^k_{(M^*,\emptyset)}$. Otherwise, we show that $\bigvee_{(M,w)\in\mathcal{D}}\chi_{w}^k\dashv\vdash\bigvee_{(M,w)\in\mathcal{D}}\theta_{\{w\}}^k=\bigvee_{(M,w)\in\mathcal{D}}(\chi_{w}^k\land\top\subseteq\chi_{w}^k)$. The direction $\dashv$ is easy. For the other direction $\vdash$, we derive $\chi_{w}^k\vdash\chi_{w}^k\land \top\subseteq\chi_{w}^k$ by Lemma \ref{lem:top_atom_derivability_results} \ref{lem:top_atom_derivability_results_i}; the result then follows by $\lor$E and $\lor$I.

\item[\boldmath$\cdot$]Let $\phi=\psi_1\lor\psi_2$ and let $k\geq md(\phi)$. Then $k\geq md(\psi_1), md(\psi_2)$, so by the induction hypothesis there are nonempty $\mathcal{C},\mathcal{D}$ such that $\psi_1\dashv\vdash\bigvee_{(M,T)\in \mathcal{C}}\theta_{T}^k\quad\text{and}\quad \psi_2\dashv\vdash\bigvee_{(M^\prime,S)\in \mathcal{D}}\theta_{S}^k$. Clearly $\psi_1\vee\psi_2 \dashv \vdash \bigvee_{(M'',U)\in \mathcal{C}\cup \mathcal{D}}\theta_{U}^k$.
 
\item[\boldmath$\cdot$]Let $\phi=\psi_1\land\psi_2$  and let $k\geq md(\phi)$. Then $k\geq md(\psi_1), md(\psi_2)$, so by the induction hypothesis there are nonempty $\mathcal{C},\mathcal{D}$ such that $\psi_1\dashv\vdash\bigvee_{(M,T)\in \mathcal{C}}\theta_{T}^k\quad\text{and}\quad \psi_2\dashv\vdash\bigvee_{(M^\prime,S)\in \mathcal{D}}\theta_{S}^k$. Let
$$\mathcal{Y}=\{\biguplus \mathcal{C}^\prime\mid \mathcal{C}^\prime\subseteq \mathcal{C} \text{ and } \biguplus \mathcal{C}^\prime \leftrightarroweq_k \biguplus \mathcal{D}^\prime, \text{ for some }\mathcal{D}^\prime\subseteq \mathcal{D}\},$$
that is, $\mathcal{Y}$ contains each disjoint union of (models with) teams in $\mathcal{C}$ which is $k$-bisimilar to some disjoint union of teams in $\mathcal{D}$. By union closure and bisimulation invariance, both $\psi_1$ and $\psi_2$ will hold in each team in $\mathcal{Y}$, and it is also easy to see that each team in which both $\psi_1$ and $\psi_2$ hold must be $k$-bisimilar to some team in $\mathcal{Y}$. We show $\psi_1\land\psi_2\dashv\vdash\bigvee_{(M'',X)\in \mathcal{Y}}\theta_X^k$.

($\vdash$) By Lemma \ref{lemma:inclusion_disjunction_elimination_generalization} \ref{lemma:inclusion_disjunction_elimination_generalization_non_modal} it suffices to show

\begin{align*}
&\bigvee_{(M,T)\in \mathcal{C^\prime}}\bigvee_{w\in T}\chi_{w}^k,\bigwedge_{(M,T)\in \mathcal{C^\prime}}\bigwedge_{w\in T}(\top\subseteq\chi_{w}^k),\bigvee_{(M^\prime,S)\in \mathcal{D^\prime}}\bigvee_{w\in S}\chi_{w}^k,\bigwedge_{(M^\prime,S)\in \mathcal{D^\prime}}\bigwedge_{w\in S}(\top\subseteq\chi_{w}^k)\\
&\vdash\bigvee_{(M'',X)\in \mathcal{Y}}\theta_X^k,
\end{align*}
for all nonempty $\mathcal{C^\prime}\subseteq\mathcal{C}$ and all nonempty $\mathcal{D^\prime}\subseteq\mathcal{D}$. This reduces to showing $\theta_{\biguplus \mathcal{C}^\prime},\theta_{\biguplus \mathcal{D}^\prime}\vdash\bigvee_{(M'',X)\in \mathcal{Y}}\theta_X^k$ for all nonempty $\mathcal{C^\prime}\subseteq\mathcal{C}$ and all nonempty $\mathcal{D^\prime}\subseteq\mathcal{D}$. For a given $\mathcal{C^\prime}$ and $\mathcal{D^\prime}$, if $\biguplus \mathcal{C^\prime} \not\leftrightarroweq_k \biguplus \mathcal{D^\prime}$, then $\theta_{\biguplus \mathcal{C^\prime}},\theta_{\biguplus \mathcal{D^\prime}}\vdash\bot\vdash\bigvee_{(M'',X)\in \mathcal{Y}}\theta_X^k$ by Lemma \ref{lemma:non_bisimilar_hintikka_formulas_contradictory} \ref{lemma:non_bisimilar_hintikka_formulas_contradictory_team}.
If $\biguplus \mathcal{C^\prime} \leftrightarroweq_k \biguplus \mathcal{D^\prime}$, then $\biguplus \mathcal{C^\prime}\in \mathcal{Y}$, whence $\theta_{\biguplus \mathcal{C^\prime}}\vdash\bigvee_{(M'',X)\in \mathcal{Y}}\theta_X^k$ by $\lor$I.

($\dashv$) Let $\biguplus C'\in \mathcal{Y}$. By Lemma \ref{lemma:disjoint_union_provability}, $\theta_{\biguplus \mathcal{C}^\prime}^k\vdash\bigvee_{(M,T)\in \mathcal{C}^\prime}\theta_{T}^k$, and by $\vee$I,  $\bigvee_{(M,T)\in \mathcal{C}^\prime}\theta_{T}^k\vdash\bigvee_{(M,T)\in \mathcal{C}}\theta_{T}^k\vdash\psi_1$. Similarly $\theta_{\biguplus \mathcal{C}'}^k\vdash \psi_2$, so by $\vee$E, $\bigvee_{\biguplus \mathcal{C}'\in \mathcal{Y}}\theta_{\biguplus \mathcal{C}'}^k\vdash \psi_1\land \psi_2$.

\item[\boldmath$\cdot$]Let $\phi=\mathsf{a}\subseteq\mathsf{b}$, where $\mathsf{a}=\alpha_1\dots\alpha_n$ and $\mathsf{b}=\beta_1\dots\beta_n$, and let $k\geq md(\phi)$. By $\subseteq$Rdt and $\subseteq$Ext,
\begin{align*}
    &\mathsf{a}\subseteq\mathsf{b} \dashv\vdash \bigwedge_{\mathsf{x}\in\{\top,\bot\}^{|\mathsf{a}|}}(\neg\mathsf{a}^\mathsf{x}\lor \mathsf{x}\subseteq\mathsf{b}).
    \end{align*}
Given the induction cases for $\ML$-formulas, conjunction and disjunction, it therefore suffices to show that each primitive inclusion atom is provably equivalent to a formula in the normal form. 
We show that $\mathsf{x}\subseteq\mathsf{b}\dashv\vdash\bigvee_{(M,T)\in\mathcal{Y}}\theta^k_{T}$, where
\begin{equation*}
   \mathcal{Y}=\{(M,T) \mid \exists w\in T\text{ such that }w\models\mathsf{b}^\mathsf{x}\}
\end{equation*}
Note that the provable equivalence we wish to show corresponds to the semantic fact in Proposition \ref{prop:inclfact}. 

($\vdash$) Let $\mathcal{M}:=\{(M',w)\mid w\models \top\}$. Then $\models\bigvee_{(M',w)\in\mathcal{M}}\chi_{w}^k$ so that $\vdash\bigvee_{(M',w)\in\mathcal{M}}\chi_{w}^k$ by Proposition \ref{prop:classical_completeness}. We have $\chi_{w}^k\vdash\bigvee_{(M',w)\in \mathcal{M}}(\chi_{w}^k\land\top\subseteq\chi_{w}^k)$ by Lemma \ref{lem:top_atom_derivability_results} \ref{lem:top_atom_derivability_results_i} and $\lor$I, so $\vdash\bigvee_{(M',w)\in \mathcal{M}}(\chi_{w}^k\land\top\subseteq\chi_{w}^k)$ by $\lor$E. To show $\bigvee_{(M',w)\in \mathcal{M}}(\chi_{w}^k\land\top\subseteq\chi_{w}^k), \mathsf{x}\subseteq\mathsf{b}\vdash\bigvee_{(M,T)\in\mathcal{Y}}\theta^k_{T}$, by Lemma \ref{lemma:inclusion_disjunction_elimination_generalization} \ref{lemma:inclusion_disjunction_elimination_generalization_non_modal} it suffices to show that for all nonempty teams $S$,
\begin{equation*}
    \bigvee_{w\in S}\chi^k_{w}, \bigwedge_{w\in S}(\top\subseteq\chi^k_{w}), \mathsf{x}\subseteq\mathsf{b}\vdash \bigvee_{(M,T)\in\mathcal{Y}}\theta^k_{T}.
\end{equation*}
If $S\in\mathcal{Y}$, the result follows by $\lor$I. If $S\not\in\mathcal{Y}$, then for any $w\in S$, $w\models\neg\mathsf{b}^\mathsf{x}$. We have that $md(\neg\mathsf{b}^\mathsf{x})=md(\mathsf{x}\subseteq\mathsf{b})\leq md(\mathsf{a}\subseteq\mathsf{b})=md(\phi)\leq k$, whence $\chi^k_{w}\vdash\neg\mathsf{b}^\mathsf{x}$ by Lemma \ref{lemma:world_satisfaction_to_hintikka_provability}. Therefore $\bigvee_{w\in S}\chi^k_{w}\vdash\neg\mathsf{b}^\mathsf{x}$ by $\lor$E. By $\subseteq_\neg$E, we have $\neg\mathsf{b}^\mathsf{x},\mathsf{x}\subseteq\mathsf{b}\vdash\bigvee_{(M,T)\in\mathcal{Y}}\theta^k_{T}$.

($\dashv$) Let $(M,T)\in \mathcal{Y}$. Then there is some $w\in T$ such that $w\models \mathsf{b}^\mathsf{x}$ so that since $md(\mathsf{b}^\mathsf{x})=md(\mathsf{x}\subseteq\mathsf{b})\leq md(\mathsf{a}\subseteq\mathsf{b})=md(\phi)\leq k$, we have $\chi^k_w\vdash \mathsf{b}^\mathsf{x}$ by Lemma \ref{lemma:world_satisfaction_to_hintikka_provability}. We have $\theta_{T}^k\vdash\top\subseteq\chi_{w}^k$ by $\land$E, and $\top\subseteq\chi_{w}^k\vdash \top\subseteq\mathsf{b}^\mathsf{x}$ by $\chi^k_w\vdash \mathsf{b}^\mathsf{x}$ and Lemma \ref{lem:top_atom_derivability_results} \ref{lem:top_atom_derivability_results_ii}. Then:
\begin{align*}
    \top\subseteq\mathsf{b}^\mathsf{x}&\vdash\top^{|\mathsf{b}|}\subseteq\mathsf{b}^\mathsf{x}\dots\mathsf{b}^\mathsf{x}  &\tag{Prop. \ref{prop:derivability_results} \ref{prop:derivability_results_weakening}} \\
    &\vdash\top^{|\mathsf{b}|}\subseteq\beta_1^{x_1}\dots\beta_n^{x_n}  &\tag{Lemma \ref{lem:top_atom_derivability_results} \ref{lem:top_atom_derivability_results_ii}}\\
    &\vdash x_1\dots x_n\subseteq\beta_1\dots\beta_n. &\tag{Lemma \ref{lem:top_atom_derivability_results} \ref{lem:top_atom_derivability_results_bot_top_exc}}\
    \end{align*}
    where $\mathsf{x}\subseteq\mathsf{b}=x_1\dots x_n\subseteq\beta_1\dots\beta_n$. Therefore $\bigvee_{(M,T)\in\mathcal{Y}}\theta^k_{T}\vdash\mathsf{x}\subseteq\mathsf{b}$ by $\lor$E.

\item[\boldmath$\cdot$]Let $\phi=\Diamond\psi$ and let $k\geq md(\phi)$. Then $n:=k-1\geq md(\psi)$ and by the induction hypothesis there is a nonempty $\mathcal{D}$ such that $\psi\dashv\vdash\bigvee_{(M,T)\in \mathcal{D}}\theta_{T}^n.$ By $\Diamond$Mon, $\Diamond\psi\dashv\vdash\Diamond\bigvee_{(M,T)\in \mathcal{D}}\theta_{T}^n$. 
We show that 
\begin{equation}\label{miracle_diamond}
\Diamond\bigvee_{(M,T)\in \mathcal{D}}\theta_{T}^n\dashv\vdash\bigvee_{(M,T)\in \mathcal{D}}\Diamond\theta_{T}^n\dashv\vdash\bigvee_{(M,T)\in \mathcal{D}}(\Diamond\bigvee_{w\in T}\chi_{w}^n\land\bigwedge_{w\in T}(\top\subseteq\Diamond\chi_{w}^n)).
\end{equation}
The modal depth of the formula on the very right is $\leq n+1=k$, so the result then follows by the induction cases for $\ML$-formulas, inclusion atoms, conjunction, and disjunction. 

The first equivalence in (\ref{miracle_diamond}) follows from the more general equivalence $\Diamond \phi\vee\Diamond \psi\dashv\vdash \Diamond(\phi\vee\psi)$, whose direction $\dashv$ can be derived by $\Diamond\lor$Distr, and the converse direction $\vdash$ by applying $\lor$E to $\Diamond\phi\vdash\Diamond (\phi\vee\psi)$ and $\Diamond\psi\vdash\Diamond (\phi\vee\psi)$ (which are given by $\Diamond$Mon). 

For the second equivalence in (\ref{miracle_diamond}), by $\lor$I and $\lor$E, it suffices to derive $\Diamond \theta_T^n\dashv\vdash \Diamond\bigvee_{w\in T}\chi_{w}^n\land\bigwedge_{w\in T}(\top\subseteq\Diamond\chi_{w}^n)$ for each $T\in \mathcal{D}$, i.e.,
\begin{equation}\label{miracle_diamond2}
\Diamond(\bigvee_{w\in T}\chi_{w}^n\land\bigwedge_{w\in T}(\top\subseteq\chi_{w}^n))\dashv\vdash \Diamond\bigvee_{w\in T}\chi_{w}^n\land\bigwedge_{w\in T}(\top\subseteq\Diamond\chi_{w}^n).
\end{equation}
Intuitively, if a team $S$ satisfies the formula on the right in (\ref{miracle_diamond2}), this means that $S$ has a successor team that is a subset (modulo bisimulation) of $T$ (captured by the left conjunct),
and that all elements in $T$ can be seen from $S$ (captured by the right conjunct). Combining these facts, one gets that $S$ has the team $T$ (modulo bisimulation) as a successor team---which is what the formula $\Diamond \theta_T^n$ on the left of the equivalence expresses.

Now, the direction $\vdash$ of (\ref{miracle_diamond2}) can be derived by:
\begin{align*}
\Diamond(\bigvee_{w\in T}\chi_{w}^n\land\bigwedge_{w\in T}(\top\subseteq\chi_{w}^n)) &\vdash\Diamond\bigvee_{w\in T}\chi_{w}^n\land\bigwedge_{w\in T}\Diamond(\top\subseteq\chi_{w}^n)  &\tag{$\Diamond$Mon}\\
&\vdash\Diamond\bigvee_{w\in T}\chi_{w}^n\land\bigwedge_{w\in T}(\top\subseteq\Diamond\chi_{w}^n),  &\tag{$\Diamond_\subseteq$Distr}
\end{align*}
while the other direction $\dashv$ is derived by:  
\begin{align*}
\Diamond\bigvee_{w\in T}\chi_{w}^n\land\bigwedge_{w\in T}(\top\subseteq\Diamond\chi_{w}^n)&\vdash \Diamond((\bigvee_{w\in T}\chi_{w}^n\vee \bigvee_{w\in T}\chi_{w}^n)\land\bigwedge_{w\in T}(\top\subseteq\chi_{w}^n))   \tag{$\subseteq_\Diamond$Distr}\\
&\vdash \Diamond(\bigvee_{w\in T}\chi_{w}^n\land\bigwedge_{w\in T}(\top\subseteq\chi_{w}^n)).
\tag{$\Diamond$Mon}
\end{align*}



\item[\boldmath$\cdot$]Let $\phi=\Box\psi$ and let $k\geq md(\phi)$. Then $n=k-1\geq md(\psi)$, so by the induction hypothesis there is a nonempty $\mathcal{D}$ such that $\psi\dashv\vdash\bigvee_{(M,T)\in \mathcal{D}}\theta_{T}^n$. By $\Box$Mon, we have that 
$\Box\psi\dashv\vdash\Box\bigvee_{(M,T)\in \mathcal{D}}\theta_{T}^n$. We show that
$$\Box\bigvee_{(M,T)\in \mathcal{D}}\theta_{T}^n\dashv\vdash\bigvee_{\mathcal{C}\subseteq\mathcal{D}}(\Box\bigvee_{w\in\biguplus \mathcal{C}}\chi_{w}^n\land\bigwedge_{w\in\biguplus \mathcal{C}}(\top\subseteq\Diamond\chi_{w}^n)).$$

The modal depth of the formula on the right is $\leq n+1=k$, so the result then follows by the induction cases for $\ML$-formulas, inclusion atoms, conjunction, and disjunction. Intuitively, if a team $S$ satisfies the formula on the left of the above equivalence, this means that $R[S]$ is bisimilar to the (disjoint) union of some teams in $\mathcal{D}$, i.e., $R[S]\leftrightarroweq_{n}\biguplus \mathcal{C}$ for some $\mathcal{C}\subseteq \mathcal{D}$ (see Lemma \ref{lemma:nf_disjoint_union}). In other words, by the forth condition, $R[S]$ is bisimilar to a subteam of $\biguplus \mathcal{C}$ (which is captured by the left conjunct of the disjunct corresponding to $\mathcal{C}$ of the formula on the right of the equivalence), and, by the back condition, each world in $\biguplus \mathcal{C}$ can be seen from $S$ (which is captured by the right conjunct of the relevant disjunct).

($\vdash$) 
By Lemma \ref{lemma:inclusion_disjunction_elimination_generalization} \ref{lemma:inclusion_disjunction_elimination_generalization_modal} it suffices to show
\begin{equation*}
    \Box\bigvee_{(M,T)\in\mathcal{E}}\bigvee_{w\in T}\chi_{w}^n,\bigwedge_{(M,T)\in\mathcal{E}}\bigwedge_{w\in T}(\top\subseteq\Diamond\chi_{w}^n)\vdash\bigvee_{\mathcal{C}\subseteq\mathcal{D}}(\Box\bigvee_{w\in\biguplus \mathcal{C}}\chi_{w}^n\land\bigwedge_{w\in\biguplus \mathcal{C}}(\top\subseteq\Diamond\chi_{w}^n)),
\end{equation*}
for all nonempty $\mathcal{E}\subseteq\mathcal{D}$. This reduces to showing
\begin{equation*}
    \Box\bigvee_{w\in\biguplus\mathcal{E}}\chi_{w}^n,\bigwedge_{w\in\biguplus\mathcal{E}}(\top\subseteq\Diamond\chi_{w}^n)\vdash\bigvee_{\mathcal{C}\subseteq\mathcal{D}}(\Box\bigvee_{w\in\biguplus \mathcal{C}}\chi_{w}^n\land\bigwedge_{w\in\biguplus \mathcal{C}}(\top\subseteq\Diamond\chi_{w}^n)),
\end{equation*}
for all nonempty $\mathcal{E}\subseteq\mathcal{D}$, which is given by $\lor$I. 

($\dashv$) Let $\mathcal{C}\subseteq\mathcal{D}$. We have:
\begin{align*}
    &&&\Box\bigvee_{w\in\biguplus \mathcal{C}}\chi_{w}^n \land \bigwedge_{w\in\biguplus \mathcal{C}}(\top\subseteq\Diamond\chi_{w}^n) \\
    &\vdash  &&\Box\bigvee_{w\in\biguplus \mathcal{C}}\chi_{w}^n \land\bigwedge_{w\in\biguplus \mathcal{C}}\Box(\top\subseteq\chi_{w}^n)\tag{$\Diamond\Box_\subseteq$Exc}\\
    &\vdash  &&\Box(\bigvee_{w\in\biguplus \mathcal{C}}\chi_{w}^n \land\bigwedge_{w\in\biguplus \mathcal{C}}(\top\subseteq\chi_{w}^n))\tag{$\Box$Mon}\\
    &=&&\Box \theta^n_{\biguplus\mathcal{C}}\\
    &\vdash &&\Box \bigvee_{(M,T)\in\mathcal{C}}\theta_{T}^n\tag{Lemma \ref{lemma:disjoint_union_provability}, $\Box$Mon}\\
    &\vdash &&\Box \bigvee_{(M,T)\in\mathcal{D}}\theta_{T}^n\tag{$\vee$I, $\Box$Mon}
\end{align*}
The result then follows by $\vee$E.\qedhere
\end{enumerate}
\end{proof}

\subsection{$\ML(\triangledown)$ and $\ML(\DotDiamond)$}

By adapting relevant rules from the system for modal inclusion logic $\ML(\subseteq)$, we obtain sound and complete systems for the two expressively equivalent might-operator logics $\ML(\triangledown)$ and $\ML(\DotDiamond)$.

\begin{definition}
The natural deduction system for $\ML(\triangledown)/\ML(\DotDiamond)$ consists of the classical rules in Table \ref{table:classical_connectives} and the $\triangledown/\DotDiamond$-rules in Table \ref{table:might_system}.
\end{definition}

\begin{table}[h]\centering
  \renewcommand*{\arraystretch}{3.8}
\begin{tabular*}{\linewidth}{@{\extracolsep{\fill}}|c c c|}
\hline
\hspace{.5cm}
   \AxiomC{$D_0$}\kern-2em
\noLine
\UnaryInfC{$\triangledown\phi$}
   \AxiomC{$D_1$}
\noLine
\UnaryInfC{$\triangledown\psi$}
\RightLabel{$\triangledown$Join}
\BinaryInfC{$\triangledown((\phi\lor\psi)\land\triangledown\phi\land\triangledown\psi)$}
\DisplayProof
&
\AxiomC{$D$}
\noLine
\UnaryInfC{$\triangledown\triangledown\phi$}
\RightLabel{$\triangledown$\ E}
\UnaryInfC{$\triangledown\phi$}
\DisplayProof 
& 
 \AxiomC{$D$}\kern-3em
\noLine
\UnaryInfC{$\DotDiamond(\DotDiamond\phi\land\psi)$}
\RightLabel{$\DotDiamond\land$Simpl}
\UnaryInfC{$\DotDiamond(\phi\land\psi)$}
\DisplayProof

 \\[3ex] \hline 
 
\AxiomC{$D$}\kern-5em
\noLine
\UnaryInfC{$\bullet\phi$}
\AxiomC{$[\phi]$}
\noLine
\UnaryInfC{$D_0$}
\noLine
\UnaryInfC{$\psi$}
\RightLabel{$\bullet$Mon\scriptsize(1)}
\BinaryInfC{$\bullet\psi$}
\DisplayProof
& 
\AxiomC{$D$}\kern-6em
\noLine
\UnaryInfC{$\bullet(\phi\lor\psi)$}
\RightLabel{$\bullet\lor$Distr}
\UnaryInfC{$\bullet\phi\lor\bullet\psi$}
\DisplayProof
&
\AxiomC{$D$}\kern-2em
\noLine
\UnaryInfC{$\alpha$}
\RightLabel{$\bullet$\ I}
\UnaryInfC{$\bullet\alpha$}
\DisplayProof 

\hspace{.5em}

 \AxiomC{$D_0$}\kern-1.5em
\noLine
\UnaryInfC{$\neg\alpha$}
\AxiomC{$D_1$}
\noLine
\UnaryInfC{$\bullet\alpha$}
\RightLabel{$\bullet_\neg$E}
\BinaryInfC{$\phi$}
\DisplayProof \\     
 
    \multicolumn{3}{|l|}{
\AxiomC{$D$}
\kern-2em\noLine
\UnaryInfC{$\bullet\phi\lor\psi$}

\AxiomC{$[\bullet\phi]$}
\noLine
\UnaryInfC{$D_0$\ \ }
\noLine
\UnaryInfC{$\chi\ \ $}

\AxiomC{$[\psi]$}
\noLine
\UnaryInfC{$D_1$}
\noLine
\UnaryInfC{$\chi$}

\RightLabel{$\lor_{\bullet}$E}
\TrinaryInfC{$\chi$}
\DisplayProof  
    
  \AxiomC{$D_0$}
\noLine
    \UnaryInfC{$\phi\lor\psi$}
\AxiomC{$D_1$}
\noLine    
    \UnaryInfC{$\bullet \chi_1$}
\AxiomC{$\dots$}
\noLine  
\UnaryInfC{}
\noLine  
\UnaryInfC{}
\noLine  
\UnaryInfC{$\dots$}
\AxiomC{$D_n$}
\noLine    
    \UnaryInfC{$ \bullet \chi_n$}
\RightLabel{$\bullet$Distr}
    \QuaternaryInfC{$((\phi\lor \chi_1\lor\dots\lor\chi_n)\land\bullet\chi_1\land\dots\land\bullet \chi_n)\lor\psi$}
    \DisplayProof}\\ 
\multicolumn{3}{|l|}{(1) The undischarged assumptions in $D_0$ are $\ML$ formulas. }\\ 

    \hline
 
\kern-2em\AxiomC{$D_0$}
    \noLine
    \UnaryInfC{$\bullet\Diamond\psi$}
    \AxiomC{$D_1$}
    \noLine
    \UnaryInfC{$\Box\bullet\phi$}
    \RightLabel{$\Box\Diamond_{\bullet}$Exc}
    \BinaryInfC{$\bullet\Diamond\phi$}
    \DisplayProof 
    
&
 \kern-2em\AxiomC{$D$}
\noLine
\UnaryInfC{$\bullet\Diamond\phi$}
\RightLabel{$\Diamond\Box_{\bullet}$Exc}
\UnaryInfC{$\Box\bullet\phi$}
\DisplayProof 

&

\kern-2em\AxiomC{$D$}
\noLine
\UnaryInfC{$\Box(\bullet\phi\lor\psi)$}

\AxiomC{$[\bullet\Diamond\phi]$}
\noLine
\UnaryInfC{$D_0$}
\noLine
\UnaryInfC{$\chi$}

\AxiomC{$[\Box\psi]$}
\noLine
\UnaryInfC{$D_1$}
\noLine
\UnaryInfC{$\chi$}

\RightLabel{$\Box\lor_{\bullet}$E}
\TrinaryInfC{$\chi$}
\DisplayProof \\

\multicolumn{3}{|l|}
{\AxiomC{$D$}
\noLine
\UnaryInfC{$\Diamond\bullet\phi$}
\RightLabel{$\Diamond_{\bullet}$Distr}
\UnaryInfC{$\bullet\Diamond\phi$}
\DisplayProof

\AxiomC{$D_0$}
\noLine
\UnaryInfC{$\Diamond\phi$}
\AxiomC{$D_1$}
\noLine    
    \UnaryInfC{$\bullet\Diamond\psi_1$}
\AxiomC{$\dots$}
\noLine  
\UnaryInfC{}
\noLine  
\UnaryInfC{}
\noLine  
\UnaryInfC{$\dots$}
\AxiomC{$D_n$}
\noLine    
    \UnaryInfC{$ \bullet\Diamond\psi_n$}
  
\RightLabel{$\bullet_\Diamond$Distr}
  \QuaternaryInfC{$\Diamond((\phi\lor\psi_1\lor\dots\lor\psi_n)\land\bullet\psi_1\land\dots\land\bullet\psi_n)$}
\DisplayProof} \\[5ex]
\hline
\end{tabular*}
 \caption{Rules for $\ML(\bullet)$ with $\bullet\in\{\triangledown,\DotDiamond\}$.}
\label{table:might_system}
\end{table}

Most rules in Table \ref{table:might_system} correspond to rules or derivable results for the $\ML(\subseteq)$-system applied to inclusion atoms of the form $\top\subseteq \alpha$ (recall the equivalence $\top\subseteq\alpha\equiv\bullet\alpha$, where $\bullet \in \{\triangledown,\DotDiamond\}$). Given that there is no syntactic restriction on what may appear in the scope of $\bullet$ (unlike with inclusion atoms of the corresponding form), these rules may now be generalized to also apply to formulas $\bullet\phi$ where $\phi$ is non-classical; this has been done whenever the generalization in question is sound.
The rules $\bullet_{\lnot}$E, $\lor_{\bullet}$E, $\bullet$Distr, $\Box\Diamond_\bullet$Exc, $\Diamond \Box_\bullet$Exc, $\Box\vee_\bullet$E, $\Diamond_\bullet$Distr, and $\bullet_\Diamond$Distr correspond to the rules $\subseteq_\neg$E, $\lor_\subseteq$E, $\subseteq$Distr, $\Box\Diamond_\subseteq$E, $\Diamond\Box_\subseteq$E, $\Box\vee_{\subseteq}$E, $\Diamond_\subseteq$Distr, and $\subseteq_\Diamond$Distr, respectively. The rule $\bullet$I corresponds to Lemma \ref{lem:top_atom_derivability_results} \ref{lem:top_atom_derivability_results_i}, and $\bullet$Mon (restricted to classical formulas) corresponds to Lemma \ref{lem:top_atom_derivability_results} \ref{lem:top_atom_derivability_results_ii}.

For both systems, we add a rule $\bullet\vee$Distr asserting the distributivity of $\bullet$ over $\vee$. The rules $\triangledown$Join and $\DotDiamond\land$Simpl reflect the entailments pointed out in Section \ref{section:preliminaries} and serve to differentiate the systems. The $\DotDiamond$-version of the rule $\triangledown$E 
is also sound, and it is derivable using $\DotDiamond\land$Simpl and $\DotDiamond$Mon.
\begin{theorem}[Soundness]\label{soundness_might}
If $\Gamma\vdash\phi$, then $\Gamma\models\phi$.
\end{theorem}
\begin{proof}
Most cases are analogous to those for $\ML(\subseteq)$. We only prove some of the more interesting cases. By the empty team property, it suffices to check soundness for an arbitrary nonempty team $T$.
\begin{enumerate}[align=left]
    \item[($\bullet$Mon)] Let $T\models\bullet\phi$ and $T\models\gamma$ for all $\gamma\in \Gamma$, and assume that $\Gamma,\phi\models\psi$, where $\Gamma$ consists of $\ML$-formulas. Then there is a nonempty (singleton) subteam $T^\prime\subseteq T$ such that $T^\prime\models\phi$. By downward closure of the formulas in $\Gamma$, we have $T^\prime\models\gamma$ for all $\gamma\in \Gamma$. It follows that $T^\prime\models\psi$, and hence $T\models\bullet\psi$.
            
    \item[($\bullet\lor$Distr)] Let $T\models\bullet(\phi\lor\psi)$. Then there is a nonempty (singleton) subteam $T^\prime\subseteq T$ such that $T^\prime\models\phi\lor\psi$. Then there are $T_1,T_2\subseteq T^\prime$ such that $T_1\cup T_2= T^\prime$, $T_1\models\phi$ and $T_2\models\psi$. W.l.o.g. suppose that $T_1$ is nonempty. Then $T\models\bullet\phi$ so that $T\models\bullet\phi\lor\bullet\psi$.
    
    \item[($\triangledown$E)] Let $T\models\triangledown\triangledown\phi$. Then there are nonempty subteams  $T_1, T_2\subseteq T$ such that $T_1\subseteq T_2\subseteq T$ and $T_1\models\phi$. Thus $T\models\triangledown\phi$.
    
    \item[($\triangledown$Join)] Let $T\models\triangledown\phi\land\triangledown\psi$. Then there are nonempty subteams $T_1,T_2\subseteq T$ such that $T_1\models\phi$ and $T_2\models \psi$. Then $T^\prime=T_1\cup T_2$ is a nonempty subteam of $T$ such that $T^\prime\models \phi\lor\psi$, $T^\prime\models\triangledown\phi$ and $T^\prime\models\triangledown\psi$. Thus $T\models\triangledown((\phi\lor\psi)\land\triangledown\phi\land\triangledown\psi)$.

    \item[($\DotDiamond\land$Simpl)] Let $T\models\DotDiamond(\DotDiamond\phi\land\psi)$. Then there is a $w\in T$ such that $\{w\}\models\psi$ and $\{w\}\models\DotDiamond\phi$, from which it follows that $\{w\}\models\phi$. Hence $T\models\DotDiamond(\phi\land\psi)$.\qedhere
\end{enumerate}
\end{proof}

\begin{proposition} \label{lemma:der_MLmight}\
\begin{enumerate}[label=(\roman*)]
    \item \label{lemma:der_MLmight_i} $\bullet\phi\lor\bullet\psi\dashv\vdash\bullet(\phi\lor\psi)$.
    \item \label{lemma:der_MLmight_ii} $\bullet(\phi\land\psi)\vdash
\bullet\phi\land\bullet\psi$.
\end{enumerate}
\end{proposition}
\begin{proof}
   The direction $\vdash$ in item \ref{lemma:der_MLmight_i} follows by $\lor$E, $\bullet$Mon and $\lor$I, the other direction is by $\bullet\lor$Distr. Item \ref{lemma:der_MLmight_ii} follows by $\bullet$Mon, $\land$E and $\land$I.
\end{proof}

We show completeness using the same strategy as used for $\ML(\subseteq)$. Many of the details are analogous to parts of the $\ML(\subseteq)$-proof; we omit most of these details and focus on the steps that specifically concern the might-operators. 


\begin{lemma}[Provable equivalence of the normal form]\label{nfproveeqMLDD}
For any formula $\phi$ in $\ML(\bullet)$ and $k\geq md(\phi)$, there is a (finite, nonempty) property $\mathcal{C}$ such that
$$\phi \dashv\vdash\bigvee_{(M,T)\in\mathcal{C}}(\bigvee_{w\in T}\chi^k_{w}\land\bigwedge_{w\in T}\bullet\chi^k_{w}) $$
\end{lemma} 
\begin{proof}
By induction on $\phi$. All inductive cases except the one for $\bullet$ are similar to the corresponding proofs for $\ML(\subseteq)$ (see Lemma \ref{nfproveeq}). Inspecting the proof of Lemma \ref{nfproveeq}, one sees that with the exception of the case for inclusion atoms, the lemma is derivable using rules which have analogues in the $\ML(\bullet)$-systems, together with the proof-theoretic results in Proposition \ref{prop:classical_completeness} and Lemmas \ref{lemma:hintikka_provable_equivalence}, \ref{lemma:team_hintikka_provable_equivalence}, \ref{lem:top_atom_derivability_results} \ref{lem:top_atom_derivability_results_i}, \ref{lem:top_atom_derivability_results} \ref{lem:top_atom_derivability_results_ii}, \ref{lemma:disjoint_union_provability}, \ref{lemma:inclusion_disjunction_elimination_generalization}, \ref{lemma:non_bisimilar_hintikka_formulas_contradictory}, and \ref{lemma:world_satisfaction_to_hintikka_provability}. These results, in turn, are also (with the exception of Lemmas \ref{lem:top_atom_derivability_results} \ref{lem:top_atom_derivability_results_i}, \ref{lem:top_atom_derivability_results} \ref{lem:top_atom_derivability_results_ii}) derivable using rules which have analogues in the $\ML(\bullet)$-system. Any steps making use of Lemma \ref{lem:top_atom_derivability_results} \ref{lem:top_atom_derivability_results_i} can be replaced by applications of $\bullet$I; similarly with Lemma \ref{lem:top_atom_derivability_results} \ref{lem:top_atom_derivability_results_ii} and $\bullet$Mon. We add the cases for $\bullet$.

Let $\phi=\bullet\psi$. We derive 
\begin{align*}
    \bullet\psi &\dashv\vdash \bullet\bigvee_{(M,T)\in \mathcal{C}}(\bigvee_{w\in T}\chi_{w}^k\land\bigwedge_{w\in T}\bullet\chi_{w}^k) \tag{Induction hypothesis, $\bullet$ Mon} \\
    &\dashv\vdash \bigvee_{(M,T)\in \mathcal{C}}\bullet(\bigvee_{w\in T}\chi_{w}^k\land\bigwedge_{w\in T}\bullet\chi_{w}^k) \tag{Prop. \ref{lemma:der_MLmight} \ref{lemma:der_MLmight_i}}
    \end{align*}
    We will now show:
    \begin{enumerate}[label=(\alph*)]
        \item $\bigvee_{(M,T)\in \mathcal{C}}\triangledown(\bigvee_{w\in T}\chi_{w}^k\land\bigwedge_{w\in T}\triangledown\chi_{w}^k) \dashv\vdash \bigvee_{(M,T)\in \mathcal{C}}\bigwedge_{w\in T}\triangledown\chi_{w}^k$ and
        \item $\bigvee_{(M,T)\in \mathcal{C}}\DotDiamond(\bigvee_{w\in T}\chi_{w}^k\land\bigwedge_{w\in T}\DotDiamond\chi_{w}^k) \dashv\vdash \bigvee_{(M,T)\in \mathcal{C}}\DotDiamond\bigwedge_{w\in T}\chi_{w}^k$ and
        \item every formula in the form  $\bullet\alpha$ is provably equivalent to a formula in normal form.
    \end{enumerate}
    The result will then follow by the induction cases for conjunction and disjunction.

We first show (a). By $\vee$E and $\vee$I, it suffices to show that $\triangledown(\bigvee_{w\in T}\chi_{w}^k\land\bigwedge_{w\in T}\triangledown\chi_{w}^k) \dashv\vdash \bigwedge_{w\in T}\triangledown\chi_{w}^k$ for an arbitrary $(M,T)\in \mathcal{C}$. For the direction $\vdash$, we derive $\triangledown(\bigvee_{w\in T}\chi_{w}^k\land\bigwedge_{w\in T}\triangledown\chi_{w}^k) \vdash \triangledown\bigwedge_{w\in T}\triangledown\chi_{w}^k\vdash\bigwedge_{w\in T}\triangledown\triangledown\chi_{w}^k\vdash\bigwedge_{w\in T}\triangledown\chi_{w}^k$ by $\triangledown$Mon, $\land$E, Proposition \ref{lemma:der_MLmight} \ref{lemma:der_MLmight_ii}, and $\triangledown$E. The direction $\dashv$ follows
by $\triangledown$Join. 

We now show (b). As above, it suffices to show that $\DotDiamond(\bigvee_{w\in T}\chi_{w}^k\land\bigwedge_{w\in T}\DotDiamond\chi_{w}^k) \dashv\vdash \DotDiamond\bigwedge_{w\in T}\chi_{w}^k$ for an arbitrary $(M,T)\in \mathcal{C}$. For the direction $\vdash$, we derive $\DotDiamond(\bigvee_{w\in T}\chi_{w}^k\land\bigwedge_{w\in T}\DotDiamond\chi_{w}^k) \vdash \DotDiamond\bigwedge_{w\in T}\DotDiamond\chi_{w}^k\vdash\DotDiamond\bigwedge_{w\in T}\chi_{w}^k$ by $\DotDiamond$Mon, $\land$E and $\DotDiamond\land$Simpl. For the direction $\dashv$, we note that $\bigwedge_{w\in T}\chi_{w}^k\vdash\bigvee_{w\in T}\chi_{w}^k\land \bigwedge_{w\in T}\DotDiamond\chi_{w}^k$ by $\land$E, $\lor$I, and $\DotDiamond$I. Therefore $\DotDiamond\bigwedge_{w\in T}\chi_{w}^k\vdash\DotDiamond(\bigvee_{w\in T}\chi_{w}^k\land\bigwedge_{w\in T}\DotDiamond\chi_{w}^k)$
by $\DotDiamond$Mon.

The proof of (c) is analogous to the proof of the primitive inclusion atom case in Lemma \ref{nfproveeq} (with an application of $\bullet$Mon replacing the appeal to Lemma \ref{lem:top_atom_derivability_results} \ref{lem:top_atom_derivability_results_ii}).\qedhere

\end{proof}


\begin{theorem}[Completeness]\label{CompletenessMLDD}
If $\Gamma\models\psi$, then $\Gamma\vdash\psi$.
\end{theorem}
\begin{proof}
    Similar to the proof of Theorem \ref{Completeness}. Inspecting this proof, one sees that it can be conducted using only rules which have analogues in the $\ML(\bullet)$-systems, together with the proof-theoretic results in Proposition \ref{prop:classical_completeness} and Lemmas \ref{lemma:hintikka_provable_equivalence}, \ref{lemma:team_hintikka_provable_equivalence}, \ref{lem:top_atom_derivability_results} \ref{lem:top_atom_derivability_results_i}, \ref{lem:top_atom_derivability_results} \ref{lem:top_atom_derivability_results_ii}, \ref{lemma:disjoint_union_provability} and \ref{nfproveeq}. We have proved an analogue to Lemma \ref{nfproveeq} in Lemma \ref{nfproveeqMLDD}, and as noted above, analogues to the other required results can be obtained in $\ML(\bullet)$.
\end{proof}

\section{Concluding remarks and directions for further research} \label{section:conclusion}

In this article, we have addressed a recognized gap in the literature on team-based modal logics by presenting an axiomatization for modal inclusion logic $\ML(\subseteq)$. This logic, together with modal dependence logic and modal independence logic, are commonly considered to be the core team-based modal logics.
While modal dependence logic has already been axiomatized in previous work \cite{yang2017}, modal independence logic (see \cite{Kont}) as well as propositional independence logic are still missing an axiomatization.

We also studied two other union-closed extensions of modal logic---the two might-operator logics $\ML(\triangledown)$ and $\ML(\DotDiamond)$. The logics $\ML(\subseteq)$ and $\ML(\triangledown)$ were shown to be expressively complete for the same class of properties and hence expressively equivalent in \cite{hella2015}; we reviewed and refined this result and showed that the new variant $\ML(\DotDiamond)$ with the singular might operator $\DotDiamond$ is likewise expressively complete for this class of properties. We also provided axiomatizations for $\ML(\triangledown)$ and $\ML(\DotDiamond)$. Note that one can obtain expressive completeness results, axiomatizations, and completeness proofs for the propositional variants of the might-operator logics via a straightforward adaptation of the results in this article.

All our axiomatizations are presented in natural deduction style---as opposed to the Hilbert style commonly used for modal logics---mainly because we do not have an {\em implication} connective in the languages of the logics we consider. To turn our natural deduction rules into Hilbert-style axioms, one could consider extending the languages with an implication. The implication would have to preserve union closure and the empty team property, and presumably also satisfy other desiderata for an implication such as the deduction theorem. Whether such a team-based implication exists is currently unknown. 
To better study the proof-theoretic properties of these logics, it would also be desirable to introduce sequent calculi for them. To this end, finding more elegant versions of the proof systems presented in this article might be useful. A point of difficulty could be that the logics do not admit uniform substitution, although sequent calculi have been developed for other team-based logics---see, e.g., \cite{frittella2016,ChenMa,Muller}.  

We observed in Section \ref{section:expressive_power} that with the strict semantics for the diamond, the three logics are not invariant under the notion of team bisimulation established in the literature. This raises the question whether one can formulate a relation between models with teams---a \emph{strict} team bisimulation---that would be strong enough to ensure invariance with respect to strict semantics but that would still respect the local character of modal logic and not imply first-order invariance.

As we now have a better understanding of the logical properties of these three union closed team-based modal logics, it is natural to ask about their possible applications in other fields. In connection with this, we note that certain closely related logics have recently found application in formal semantics. The two-sorted first-order team semantics framework in \cite{alonidegano} employs first-order inclusion atoms together with existential quantification to represent epistemic modalities in a manner similar to how the might operators function. In \cite{aloni2022}, the usual modal logic is extended with a \emph{nonemptiness atom} $\normalfont\textsc{ne}$ satisfied by all but the empty team; this logic (which is union closed but does not have the empty team property, and is in some ways very similar to the might-operator logics---see footnote \ref{footnote:NE}) is then used to account for \emph{free choice inferences} and related natural-language phenomena.


\begin{appendices}

\section*{Appendix
}\label{secA1}



In this appendix, we provide a translation of $\ML(\subseteq)$ into first-order inclusion logic ($FO(\subseteq)$). Before we do so, however, let us issue a note of caution. One way to utilize the standard translation of $\ML$ into classical first-order logic is to derive the compactness of $\ML$ from that of first-order logic. The third author's \cite{yang2017}, similarly, provides a translation from modal dependence logic into first-order dependence logic, and uses this translation and the fact that first-order dependence logic is compact \cite{puljujärvi2022compactness} to conclude that modal dependence logic is likewise compact. The presentation in \cite{yang2017} is erroneous because it conflates compactness formulated in terms of satisfiability (if each finite subset of $\Gamma$ is satisfiable, then $\Gamma$ is satisfiable) with compactness formulated in terms of entailment (if $\Gamma\models \phi$, then there is a finite $\Gamma_0\subseteq \Gamma$ such that $\Gamma_0\models \phi$). These notions need not coincide for a logic that is not closed under classical negation, and, indeed, while first-order dependence logic is satisfiability-compact, it is not entailment-compact. The same applies to $FO(\subseteq)$, which is shown to be satisfiability-compact in  \cite{puljujärvi2022compactness}, but which is not entailment-compact. 
\footnote{To see why, consider the $FO(\subseteq)$-sentence $\phi_{\textsf{nwf}}:=\exists x\exists y (x\subseteq y \land x <y)$ where $<$ is a binary relation symbol. Now $\phi_{\textsf{nwf}}$ expresses that there is an infinite descending $<$-chain, i.e., $<$ is not well-founded (see, e.g., \cite{yang2020} for details). Let $\phi_{\textsf{lo}}$ be a first-order sentence expressing that  $<$ is a strict linear order with a (unique) greatest element. For each $n\in \mathbb{N}$, define the first-order sentence $\phi_n:= \exists! z\forall x((x= z \lor x<z)\wedge \exists y_1\dots \exists y_n(y_n<\dots <y_1<z))$, expressing that there is a $<$-chain of length $n$ below the greatest element.
Thus $\{\phi_{\textsf{lo}},\phi_n\mid n \in \mathbb{N}\}\models \phi_{\textsf{nwf}}$, but the entailment does not hold for any finite subset of $\{\phi_{\textsf{lo}},\phi_n\mid n \in \mathbb{N}\}$.
}
 One ought not, therefore, use the translation in this section to argue that $\ML(\subseteq)$ is entailment-compact. As noted in Section \ref{section:expressive_power}, however, we can derive both the satisfiability- and the entailment-compactness of $\ML(\subseteq)$ (as well as those of its variants, and those of modal dependence logic) from the fact, proved in \cite{kontinen2015}, that team properties invariant under bounded bisimulation can be expressed in classical first-order logic.

We briefly recall the syntax and semantics of $FO(\subseteq)$; for detailed discussion of the logic, see., e.g., \cite{galliani2012}. Fix an infinite set $\mathsf{Var}$ of first-order variables, and a first-order vocabulary $\tau$. The set of first-order $\tau$-terms is defined as usual.
%
%
 The syntax for $FO(\subseteq)$  over 
 $\tau$ is given by:
\begin{align*}
   \phi ::= &t_1=t_2\mid R(t_1,\dots,t_n)\mid \vec{x}\subseteq \vec{y}\mid\neg t_1=t_2\mid \neg R(t_1,\dots,t_n)\mid\\
   &(\phi \lor \phi) \mid(\phi \land \phi) \mid 	\exists x \phi\mid \forall x\phi,
\end{align*}
where each $t_i$ is a $\tau$-term, 
 $R$ is an $n$-ary relation symbol in $\tau$, and $\vec{x}$ and $\vec{y}$ are two finite sequences of variables of the same length. 
Define, as usual, $\bot:=\forall x (\neg x=x)$ and $\top:=\forall x (x=x)$.

Let $\mathcal{M}$ be a $\tau$-model with domain $W$. An assignment over $W$ with domain $\mathsf{V}\subseteq \mathsf{Var}$ is a function $s: \mathsf{V}\to W$. We abbreviate $s(\vec{x}):=\langle s(x_1),\ldots, s(x_n)\rangle$, where $\vec{x}=\langle x_1,\ldots, x_n\rangle$. Given $a\in W$, the modified assignment 
$s(a/x)$ is defined as 
$s(a/x)(y):=a$ for $y=x$, and $s(a/x)(y):=s(y)$ otherwise. Given $\mathsf{V}\subseteq{Var}$, a (first-order) team $X$ of $\mathcal{M}$ with domain $\mathsf{V}$ is a set of assignments $s: \mathsf{V}\to W$. Given $A\subseteq W$ and a team $X$ of $\mathcal{M}$, we let $X(A/x):=\{s(a/x)\mid s\in X,a\in A\}$.

%

The team semantics of $FO(\subseteq)$ is given by (here we follow the convention for first-order team-based logics and write the team on the right-hand side of the satisfaction symbol):
\begin{align*}
%
\mathcal{M}\models_X \alpha \iff  &\text{ for all }s\in X, \mathcal{M}\models_s \alpha, \text{ if  $\alpha$  is a first-order literal;} \\
\mathcal{M}\models_X \vec{x}\subseteq \vec{y} \iff &\text{ for all } s\in X \text{ there exists } s'\in X \text{ such that }s({\vec{x}})=s'({\vec{y}});\\
\mathcal{M}\models_X \phi\lor\psi\iff &\text{ there exist }Y,Z\subseteq X\text{ such that }X = Y\cup Z, \mathcal{M}\models_Y \phi\text{  and }\mathcal{M}\models_Z \psi;\\
\mathcal{M}\models_X \phi\land\psi \iff  &\mathcal{M}\models_X \phi\text{ and }\mathcal{M}\models_X \psi;\\
\mathcal{M}\models_X \exists x \phi \iff &\text{  there exists a function }F : X \to \mathcal{P}(W)\setminus\{\emptyset\}\text{ such that }\\
&\mathcal{M}\models_{X(F/x)}\phi,\text{ where }X(F/x): = \{s(a/x) \mid s\in X, a\in  F(s)\};\\
\mathcal{M}\models_X \forall x \phi \iff  &\mathcal{M}\models_{X(W/x)}\phi
\end{align*}

It is easy to check that all formulas of $FO(\subseteq)$ are union closed and have the empty team property, and that formulas without inclusion atoms (formulas of $FO$) are additionally downward closed and flat (the first-order versions of these notions are defined analogously to the modal versions).

\begin{definition}\label{standard_translation}
For any formula $\phi$ in $\ML(\subseteq)$, its standard translation $ST_x(\phi)$ into first-order inclusion logic $FO(\subseteq)$ (with respect to the first-order variable $x$) is defined inductively as follows:
\begin{align*}
    ST_x(p)&:=Px \\
     ST_x(\bot)&:=\bot \\
      ST_x(\neg\alpha)&:= \neg ST_x(\alpha)\\
       ST_x(\phi\land\psi)&:= ST_x(\phi)\land ST_x(\psi)\\
ST_x(\phi\lor\psi)&:= ST_x(\phi)\lor ST_x(\psi) \\
    ST_x(\Diamond\phi)&:=\exists y(xRy \land ST_y(\phi)) \\
        ST_x(\Box\phi)&:=\forall y(\neg xRy \lor (xRy\land ST_y(\phi))) \\
    ST_x(\mathsf{a} \subseteq \mathsf{b})&:= \bigwedge_{\mathsf{z}\in \{\top,\bot\}^{|\mathsf{a}|}}(ST_x(\lnot \mathsf{a}^\mathsf{z})\lor \exists y(y\subseteq x \land ST_y(\mathsf{b}^\mathsf{z}))
\end{align*}
\end{definition}

For $\mathsf{X}\subseteq \mathsf{Prop}$, let $\sigma_\mathsf{X}$ be the first-order signature containing a binary relation symbol $R_0$ and a unary relation symbol $P$ for each $p\in \mathsf{X}$. There is a one-to-one correspondence between Kripke models over $\mathsf{X}$ and first-order $\sigma_\mathsf{X}$-structures: $M=(W,R,V)$ over $\mathsf{X}$ corresponds to the $\sigma_\mathsf{X}$-structure $\mathcal{M}=(W,R^{\mathcal{M}}_0,\{P^{\mathcal{M}}\}_{p\in\mathsf{X}})$, where $P^{\mathcal{M}}=V(p)$ is the interpretation of each unary $P$, and where $R^{\mathcal{M}}_0=R$.

We conclude with a lemma establishing that the translation functions as desired.

\begin{lemma}\label{lemma:compactnesslemma}
Let $\phi$ be a formula in modal inclusion logic $\ML(\subseteq)$. For any Kripke model $M$ and team $T$ of $M$, and any variable $x$, 
\begin{equation*}
M,T \models\phi\iff \mathcal{M} \models_{T_x}ST_x(\phi),    
\end{equation*}
where $T_x=\{\{(x,w)\} | w \in T\}$ is a first-order team with domain $\{x\}$.
\end{lemma}

\begin{proof}
The proof is by induction on $\phi$; we refer to \cite{luck} and add the case for $\mathsf{a}\subseteq\mathsf{b}$.

We want to show $M,T\models \mathsf{a}\subseteq \mathsf{b}\iff \mathcal{M},T_x\models \bigwedge_{\mathsf{z}\in \{\top,\bot\}^{|\mathsf{a}|}}(ST_x(\lnot \mathsf{a}^\mathsf{z})\lor \exists y(y\subseteq x \land ST_y(\mathsf{b}^\mathsf{z}))$. By Lemma \ref{lemma:inclusion_atom_reduction} \ref{lemma:inclusion_atom_reduction_i}, we have $\mathsf{a}\subseteq\mathsf{b}\equiv \bigwedge_{\mathsf{z}\in\{\top,\bot\}^{|\mathsf{a}|}}(\neg\mathsf{a}^\mathsf{z}\lor \mathsf{z}\subseteq\mathsf{b})$. It therefore suffices to show that for any $\mathsf{z}\in \{\top,\bot\}^{|\mathsf{a}|}$, $M,T\models \mathsf{z}\subseteq \mathsf{b}\iff \mathcal{M}\models_{T_x}\exists y(y\subseteq x \land ST_y(\mathsf{b}^\mathsf{z}))$, for then the result follows by the other induction cases. Let $\mathsf{z}\in \{\top,\bot\}^{|\mathsf{a}|}$.

The case in which $T=\emptyset$ is trivial. Suppose that $T\neq\emptyset$ whence also $T_x\neq \emptyset$. For the left-to-right direction, by Proposition \ref{prop:inclfact} there is a $w\in T$ such that $\{w\}\models\mathsf{b}^\mathsf{z}$. By the induction hypothesis, $\mathcal{M}\models_{\{w\}_x}ST_x(\mathsf{b}^\mathsf{z})$. It is easy to verify that this is equivalent to $\mathcal{M}\models_{T_x (\{w\}/y)}ST_y(\mathsf{b}^\mathsf{z})$. 
Define a function $F:T_x\rightarrow \mathcal{P}(W)\setminus\{\emptyset\}$ by $F(s)=\{w\}$  for all $s\in T_x$. Clearly $\mathcal{M}\models_{T_x (F/y)}ST_y(\mathsf{b}^\mathsf{z})$. Since $w\in T$, $\mathcal{M}\models_{T_x (F/y)}y\subseteq x$, whence $\mathcal{M}\models_{{T_x}}\exists y(y\subseteq x\land ST_y(\mathsf{b}^\mathsf{z}))$.

For the other direction, we have that there is a function $F:T_x\rightarrow \mathcal{P}(W)\setminus\{\emptyset\}$, such that $\mathcal{M}\models_{T_x (F/y)}(y\subseteq x\land ST_y(\mathsf{b}^\mathsf{z}))$. We have that $T_x\neq \emptyset$ whence also $T_x (F/y)\neq \emptyset$ so let $s\in T_x (F/y)$. By $\mathcal{M}\models_{T_x (F/y)}y\subseteq x$, there is a $s'\in T_x (F/y)$ with $s(y)=s'(x)$. We must have $s'(x)=w \in T$; then $s=s''(w/y)$ for some $s''\in T_x$. By downward closure, $\mathcal{M}\models_{\{s''(w/y)\}}ST_y(\mathsf{b}^\mathsf{z})$. It follows that 
$\mathcal{M}\models_{\{w\}_y}ST_y(\mathsf{b}^\mathsf{z})$. By the induction hypothesis, $M,\{w\}\models \mathsf{b}^\mathsf{z}$, so that by Proposition \ref{prop:inclfact}, $M,T\models \mathsf{z}\subseteq\mathsf{b}$.
\qedhere


\end{proof}




\end{appendices}



\bibliographystyle{sn-basic} 

\section*{Declarations}
\textbf{Competing interests} The authors have no competing interests to declare that are relevant to the content of this article.

\bibliography{sn-bibliography}

\newpage

\backmatter




\end{document}